\documentclass[reqno,11pt,oneside]{amsart}
\usepackage{amsmath,amssymb,amsthm,mathrsfs}
\usepackage{epsfig,color}
\usepackage[latin1]{inputenc}
\usepackage[english]{babel}
\usepackage[numbers]{natbib}

\usepackage{esint}

\usepackage{hyperref}

 \allowdisplaybreaks
\numberwithin{equation}{section}

\def\DD{\mathcal D}

\def\R{\mathbb R}

\def\pa{\partial}
\def\trace{{\rm tr}}

\def\00{{\bf 0}}

\newcommand{\tr}{\mbox{tr}\,}

\newcommand{\Hess }{{\rm Hess}}

\newcommand{\diver}{{\rm div \,}}

\newcommand{\divg}{{\rm div}_g \,}

\newcommand{\ric}{{\rm Ric}}

\newtheorem*{theorem*}{Theorem}
\newtheorem{theorem}{Theorem}[section]
\newtheorem{lemma}[theorem]{Lemma}
\newtheorem{proposition}[theorem]{Proposition}

\theoremstyle{definition}

\begin{document}
  
    \title[Symmetry for critical points of CKN inequalities]{Symmetry for positive critical points of Caffarelli-Kohn-Nirenberg inequalities}
    

\author{Giulio Ciraolo}
\address{G. Ciraolo. Dipartimento di Matematica "Federigo Enriques",
Universit\`a degli Studi di Milano, Via Cesare Saldini 50, 20133 Milano, Italy}
\email{giulio.ciraolo@unimi.it}

\author{Rosario Corso}
\address{R. Corso. Dipartimento di Matematica e Informatica,
	Universit\`a degli Studi di Palermo, Via Archirafi 34, 90123 Palermo, Italy
}
\email{rosario.corso02@unipa.it}

    
 \begin{abstract}
We consider positive critical points of Caffarelli-Kohn-Nirenberg inequalities and prove a Liouville type result which allows us to give a complete classification of the solutions in a certain range of parameters, providing a symmetry result for positive solutions. The governing operator is a weighted $p$-Laplace operator, which we consider for a general $p \in (1,d)$. For $p=2$, the symmetry breaking region for extremals of Caffarelli-Kohn-Nirenberg inequalities was completely characterized in \cite{DEL}. Our results extend this result to a general $p$ and are optimal in some cases.
\end{abstract}

\maketitle

\tableofcontents

\section{Introduction}
In this paper we study the symmetry of critical points related to Caffarelli-Kohn-Nirenberg (CKN) inequalities in $\R^d$, with $d \geq 3$. CKN inequalities where proved in \cite{CKN} (see also \cite{ILIN, Lin, Mazya}) and assert that there exists a positive constant $C_{a,b}$ such that 
\begin{equation}\label{CKN}
\left (\int_{\R^d} |x|^{-bq}u^q dx\right)^{\frac{1}{q}} \leq C_{a,b} \left (\int_{\R^d} |x|^{-ap}|Du|^p dx\right )^{\frac{1}{p}}
\end{equation}
holds for any $u\in \DD^{1,p}(\R^d,|x|^{-ap})$, where $\DD^{1,p}(\R^d,|x|^{-ap})$ is the completion of $C^\infty_0(\R^d)$ with respect to the norm 
$$
\|u\|_{\DD^{1,p}(\R^d,|x|^{-ap})}=\left (\int_{\R^d} |x|^{-ap}|D u|^p dx \right )^{\frac 1 p}.
$$
In this paper, we consider $d\geq 3$, $p>1$, $a\leq b <  a+1$, $q=\frac{dp}{d-p(1+a-b)}$, $a < a_c$, where $a_c=\frac{d}{p}-1$, and $C_{a,b}=C_{a,b}(d,p,a,b,q)$. We notice that the exponent $q$ is determined by the invariance of the inequality under scaling. We mention that \eqref{CKN} is sometimes called the Hardy-Sobolev inequality, since it can be seen as an interpolation between the Sobolev inequality (i.e. for $a=b=0$) and the weighted Hardy inequalities (when $b=a+1$), see \cite{CatrinaWang} for more details.

In the last decades, a large effort has been spent to investigate the sharp constants in \eqref{CKN} \cite{CNV, DPD}, as well as symmetry properties or symmetry breaking of extremals. Indeed, it is well-known that minimizers of \eqref{CKN} may not be radially symmetric for some values of $a$ and $b$ and hence symmetry breaking may occur \cite{FS, CM}. The problem of identifying the optimal symmetry breaking region is still open for a general $p$, but it has been recently settled in \cite{DEL} for $p=2$ (see also the references quoted in the introduction of \cite{DEL} for many interesting partial results). In particular, in \cite{DEL} the authors elegantly characterize the optimal symmetry range for minimizers, and more generally for positive critical points, of CKN inequalities for $p=2$ by using the so-called \emph{carr\'e du champ} method. 

In this paper we do not restrict to the case $p=2$ and we consider positive critical points of CKN inequalities \eqref{CKN}, which (up to a multiplicative constant) are solutions to
\begin{equation}
\label{eq_main_u_p}
\begin{cases}
\diver(|x|^{-ap}|D u|^{p-2}D u)+|x|^{-bq}u^{q-1}=0 \;\text{ in }\R^d,\\
u>0,\\
u\in \DD^{1,p}(\R^d,|x|^{-ap})
\end{cases}
\end{equation}
with $1<p<d$.

Our goal is to investigate the optimal region of symmetry breaking and we prove the following symmetry result for solutions to \eqref{eq_main_u_p}.

\begin{theorem}\label{th_main_p}
Let $(a,b) \neq (0,0)$ and let $u$ be a solution of \eqref{eq_main_u_p} with
		\begin{equation}\label{cond_rad}
		\frac{(1+a-b)(a_c-a)}{a_c-a+b}\leq \sqrt{ \frac{d-2}{n-2}}, 
		\end{equation}
	where $n=\frac{d}{1+a-b}$. If $a=b$ or $p<n/2$, then 
	\begin{equation} \label{u_Talenti}
	u(x)=\left( \frac{q}{p(p-1)^{p-1}}  \right)^{\frac{1}{q-p}} \left (\frac{(d-p(1+a))\lambda}{\lambda^p+c_2|x|^{\frac{(q-p)(d-p(1+a))}{p(p-1)}} }\right )^{\frac p{q-p}}
	\end{equation}
	for some $\lambda > 0$. 
\end{theorem}

As already mentioned, CKN inequalities reduce to Sobolev's inequality for $a=b=0$. When $p=2$, symmetry for positive critical points of Sobolev's inequality goes back to the fundamental papers \cite{CGS,GNN}. For a general $p \in (1,d)$, it has been proved in \cite{Sc,Vetois} and, more recently, in \cite{CFR} for general norms of $\R^d$ and in convex cones. 

We mention that the case $a=b=0$ could be included in Theorem  \ref{th_main_p}, even if in this case the solution \eqref{u_Talenti} is valid up to a translation. Indeed, as it is well-known, minimizers of Sobolev's inequality are the so-called Aubin-Talenti functions and it is easy to see that minimizers are unique up to multiplication by a constant, translations and scaling. When $a$ or $b$ do not vanish, minimizers have less degree of symmetry since the functional is invariant under rotations and reflections about the origin, and scalings. For this reason and since our approach is a generalization of \cite{CFR} to the present case, we prefer to state Theorem \ref{th_main_p} only for $(a,b) \neq (0,0)$.
 
When $(a,b) \neq (0,0)$, the presence of radial weights implies that a new crucial question must be addressed: for what parameters $a$ and $b$ is the radial symmetry of \eqref{CKN} inherited by minimizers of \eqref{CKN} or, more generally, by its critical points? Theorem \ref{th_main_p} partially answers to this question and it is optimal in some cases. Here we just mention that \eqref{cond_rad} holds for a large range of parameters, for instance whenever $a \geq 0$ and $p \geq 2$ (see also the more detailed discussion after Theorem \ref{main_thm_Riem}). As far as we know, there are some partial results on the rigidity of minimizers of \eqref{CKN} but few results are available for critical points when one considers $1<p<d$ (see \cite{Alvino1,DongLamLu,LamLu} and references therein). We also mention that some results on symmetry breaking for $p>1$ are obtained in \cite{ByeonWang, CaldiroliMusina,Nazarov,SmetsWillem}.

We notice that the approach used in this paper and the one in \cite{DEL} have in common the same starting point. When $p=2$, our approach can be compared to the one in \cite{DEL} and it can be seen that we use an equivalent reformulation of the problem in a Riemannian setting (the approach in \cite{DEL} uses a warped manifold setting). The main difference between the two approaches is that in \cite{DEL} the authors benefit of the linearity of the Laplace operator and, in particular, they make use of the Kelvin transform and of the spherical representation of the operator. Thanks to this setting, they are able to perform some further steps in the proof which lead to optimality.

\medskip

\noindent {\bf{Strategy of the proof.}  } Our approach is based on a reformulation of CKN inequalities in a suitable Riemannian manifold which gives a Sobolev type inequality on $\R^d$ with a weight $|x|^{n-d}$, and then to extend the approach contained in \cite{CFR} to the present situation.

Inspired by \cite{DEL}, we consider the change of variables $x \mapsto |x|^{\alpha-1} x$, with $\alpha \in \R$ which has to be determined later. By setting $w( |x|^{\alpha-1} x ) = u(x)$, CKN inequalities can be written as 
\begin{equation}
\label{ineq_w_pre_p}
\alpha^{\frac{1}{q}}\left( \int_{\R^d} w^q|x|^{\frac{d-bq}{\alpha}-d}dx \right )^{1/q}\leq C_{a,b}  \alpha^{\frac{1}{p}}\left( \int_{\R^d}  | A(x)D w|^p |x|^{\frac{d-ap+(\alpha-1)p}{\alpha}-d}dx\right )^{1/p},
\end{equation}
where
$$
A(x)= \left(\delta_{ij}+(\alpha-1)\frac{x_ix_j}{|x|^2}\right).
$$
We choose $\alpha$ such that 
$$
n:=\frac{d-bq}{\alpha}=\frac{d-ap+(\alpha-1)p}{\alpha},
$$
i.e. 
\begin{equation} \label{alpha_def}
\alpha=\frac{(1+a-b)(a_c-a)}{a_c-a+b}
\end{equation}
and 
\begin{equation} \label{n_def}
n=\frac{pq}{q-p}=\frac{d}{1+a-b}.
\end{equation}
Notice that $\alpha>0$. In this way, \eqref{ineq_w_pre_p} can be written as follows 

\begin{equation}
\label{ineq_w_p}
\alpha^{\frac 1 q-\frac 1 p }\left(\int_{\R^d} w^q|x|^{n-d} dx \right )^{1/q}\leq C_{a,b} \left(\int_{\R^d} | A(x)D w|^p |x|^{n-d}dx\right )^{1/p}.
\end{equation}
It will be convenient to consider $\R^d$ with the metric $g$ such that $g^{-1}= A^{T}A$, i.e. 
\begin{equation} \label{g_def}
g_{ij}=\delta_{ij}+\left (\frac 1{\alpha^2}-1\right )\frac{x_ix_j}{|x|^2}
\end{equation}
and hence
\begin{equation} \label{g_inverse}
g^{ij}=\delta_{ij}+\left (\alpha^2-1\right )\frac{x_ix_j}{|x|^2}\,.
\end{equation}
We notice that $g$ is smooth outside the origin, it is zero homogeneous and $\det(g)=\alpha^{-2}$ is constant. In this setting, CKN inequalities can be written as
\begin{equation}
\label{ineq_w_p_metric}
\left(\int_{\R^d} w^q|x|^{n-d} dV_g \right )^{1/q}\leq  C_{a,b} \left(\int_{\R^d} |\nabla w|_g^p\, |x|^{n-d}dV_g\right )^{1/p} \,,
\end{equation}
where $\nabla w$ is the gradient in the metric $g$. Inequality \eqref{ineq_w_p_metric} is the starting point for our analysis, since critical points of CKN-inequalities \eqref{ineq_w_p_metric} (and hence of \eqref{CKN}) can be seen as the solutions of
the Euler-Lagrange equation of \eqref{ineq_w_p_metric} (after an opportune re-normalization), which is given by
\begin{equation} \label{EL_w_p_II}
\frac{1}{|x|^{n-d} } \diver (|x|^{n-d}|\nabla w|_g^{p-2}\nabla w) + w^{q-1}=0 \,.
\end{equation} 
Here we recall that, since $\det(g)$ is constant, then $\diver X = \divg X$ for any vector field $X$ and for this reason we omit the dependency on $g$ in the divergence. Moreover, we look for a solution $w$ which is positive and belongs to the energy space $\mathcal{D}_g^{1,p}(\R^d,|x|^{n-d})$, i.e. such that
$$
\int_{\R^d} |\nabla w(x)|_g^p |x|^{n-d} dx < + \infty   \,.
$$ 
Notice that by CKN-inequality we also have that 
$$
\int_{\R^d} |w(x)|^q |x|^{n-d} dx < +\infty \,.
$$
Here and in the following, $Du$ denotes the Euclidean gradient and $\nabla u$ denotes the gradient in the Riemmanian manifold $(\R^d,g)$, with $g$ given by \eqref{g_def}.
Thanks to this new setting, Theorem \ref{th_main_p} is an immediate consequence of the following theorem.

\begin{theorem}\label{main_thm_Riem}
Let $(a,b)\neq (0,0)$, $d\geq 3$, $\alpha$ and $n$ be given by \eqref{alpha_def} and \eqref{n_def}, respectively, and assume that 
\begin{equation}\label{alpha_cond}
	\alpha \leq \sqrt{\frac{d-2}{n-2}} \,.
	\end{equation}
Let $w \in \mathcal{D}_g^{1,p}(\R^d,|x|^{n-d})$ be a positive solution of \eqref{EL_w_p_II}, and assume that one of the following conditions is satisfied
\begin{itemize}
\item[(H1)] $n=d$, i.e. $a=b$;  
\item[(H2)] $p<\frac{n}{2}$.
%
\end{itemize}
Then 
\begin{equation} \label{w_symmetry}
w(x)=\left( \frac{q}{p(p-1)^{p-1}}  \right)^{\frac{1}{q-p}} \left (\frac{(d-p(1+a))\lambda}{\lambda^p+c_2|x|^{\frac{p}{p-1} }}\right )^{\frac p{q-p}}
%
\end{equation}
for some $\lambda >0 $.
\end{theorem}

Now, we comment the assumptions in Theorem \ref{main_thm_Riem}. We first remark that assumptions (H1) and (H2) are probably technical, but we were not able to remove them. More precisely, due to the lack of regularity of the solution, we need the following integrability information:
\begin{equation} \label{tech_cond}
(n-d) \int_{B_1(O)} |\nabla w(x)|_g^{2(p-1)} |x|^{n-d-2} dx < +\infty.
\end{equation}
When $n=d$ this term disappears. Moreover, by using classical tools from elliptic regularity theory, in this case one can prove that $\nabla w$ is bounded, which immediately would give the desired integrability on $\nabla w$. It is expected that \eqref{tech_cond} still holds when $n-d$ is small following the approach in \cite{Stredulinsky}.
By using a scaling argument, we prove that $|\nabla w|_g \leq C/|x|$ as $x \to O$, and this implies \eqref{tech_cond} under the assumption (H2).

\begin{figure} \label{fig1}
\includegraphics[width=0.8\textwidth]{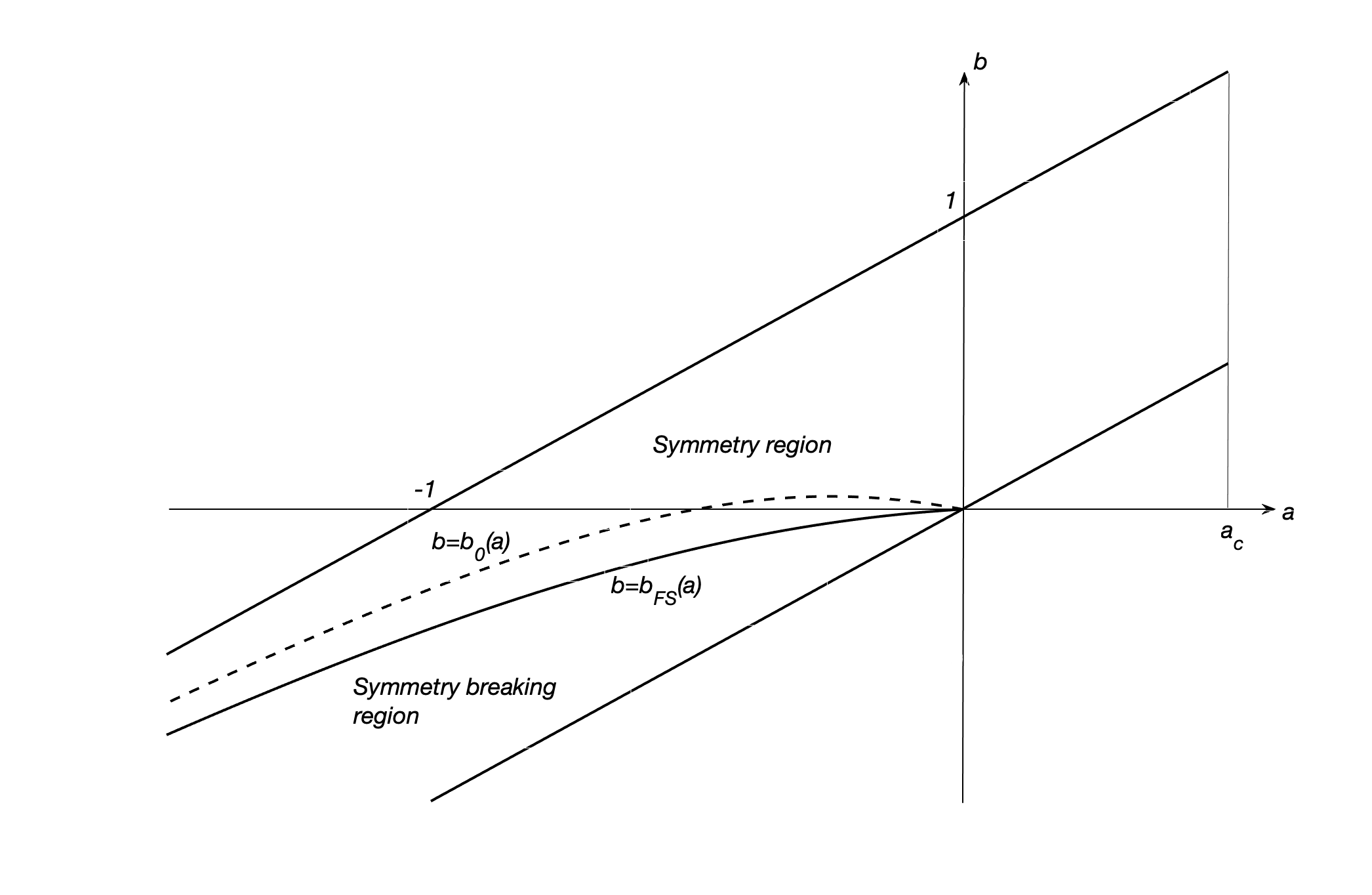}
\caption{In the case $p=2$ the optimal symmetry region is determined by the Felli-Schneider curve $b=b_{FS}(a)$. The symmetry region that corresponds to \eqref{alpha_cond} is the one between the dashed curve $b=b(a)$ and the line $b=a+1$.} 
\end{figure}

Now we comment condition \eqref{alpha_cond}. As we already mentioned, for $p=2$ it has been showed in \cite{DEL} that the optimal symmetry region is given by 
\begin{equation} \label{alpha_FS}
\alpha \leq \sqrt{\frac{d-1}{n-1}} \,.
\end{equation}
When $p=2$, the equality case in \eqref{alpha_FS} determines a curve $b=b_{FS}(a)$ (the Felli-Schneider curve) and a corresponding region where symmetry is broken, see Fig. \ref{fig1}, and it is sharp for $p=2$. The equality case in \eqref{alpha_cond} determines a curve that we denote by $b=b_0(a)$. 
We notice that \eqref{alpha_FS} and \eqref{alpha_cond} coincide for $n=d$, i.e. for  $a=b$, and hence in this case Theorem \ref{main_thm_Riem} is optimal (at least for $p=2$). For $a \neq b$, \eqref{alpha_cond} is stronger than \eqref{alpha_FS} and hence, even in the case $p=2$, we do not achieve the optimal region of symmetry for any possible range of the parameters $a$ and $b$ (see Fig. \ref{fig1}). Since the two conditions coincide for $a=b$ and comparing our proof to the one in \cite{DEL} we give the following conjecture:

\medskip

{\bf \emph{Conjecture:}} \emph{ for any $1<p<d$ the optimal symmetry region is given by \eqref{alpha_FS}}.

\medskip

%

\noindent The paper is organized as follows. In Section \ref{sect_Riem} we set the mathematical framework by providing details on the Riemannian reformulation of the problem and fix the notation. In Section \ref{sect_identity} we prove a differential identity which is at the starting point for proving Theorems \ref{th_main_p} and \ref{main_thm_Riem}. Due to the lack of regularity, the differential identity proved in Section \ref{sect_identity} must be used in an weaker integral formulation and for this reason we need some regularity results and asymptotic estimates which are proved in Sections \ref{sect_reg1} and \ref{section_reg_A}. In Section \ref{section_finalproof} we prove Theorems \ref{th_main_p} and \ref{main_thm_Riem}.

\section{The Riemannian setting} \label{sect_Riem}
In this section we give more details on the Riemmanian setting that we use in this paper. We consider a solution $u$ of \eqref{eq_main_u_p} and we change the coordinates as done for obtaining \eqref{ineq_w_pre_p}. This change of coordinates suggests to consider a new formulation of the problem in the Riemannian manifold $(\R^d,g)$, where $g$ is given by \eqref{g_def}:
\begin{equation*} 
g_{ij}=\delta_{ij}+\left (\frac 1{\alpha^2}-1\right )\frac{x_ix_j}{|x|^2} \,.
\end{equation*}
We notice that $g$ is zero-homogeneous and it has constant determinant 
$$
\det(g)=\alpha^{-2} \,.
$$
In this setting, finding a solution $u$ for \eqref{eq_main_u_p} is equivalent to find a solution $w$ (up to a multiplication constant) of the following problem
\begin{equation} \label{pb_w_Riem}
\begin{cases}
\frac{1}{|x|^{n-d} } \diver (|x|^{n-d}|\nabla w|_g^{p-2}\nabla w) + w^{q-1}=0 & \text{ in } \R^d \,,\\
w>0 \,, & \\
w \in \mathcal D^{1,p}_g(\R^d,|x|^{n-d}) \,,
\end{cases}
\end{equation} 
where $\mathcal D^{1,p}_g(\R^d,|x|^{n-d})$ is the completion of $C^\infty_0(\R^d)$ with respect to the norm 
$$
\|w\|_{\DD^{1,p}_g(\R^d,|x|^{-ap})}=\left (\int_{\R^d} |\nabla w|_g^p |x|^{n-d} dx \right )^{\frac 1 p}.
$$
Notice that, from CKN inequality \eqref{ineq_w_p_metric}, we also have that 
\begin{equation*}
\int_{\R^d} |w(x)|^q |x|^{n-d} dx < +\infty \,.
\end{equation*}
Before proceeding further, we clarify and set the notation.

\subsection{Notation}
Given a function $u:\R^d \to \R$ we denote by $Du=(\partial_j u)_{j=1,\ldots,d}$ and $D^2 u = (\partial_{ij} u)_{i,j=1,\ldots,d}$ the Euclidean gradient and the Euclidean Hessian of $u$, respectively. The Euclidean norm is denoted by $|\cdot|$ and $a\cdot b$ denotes the scalar product between two points $a$ and $b$ of $\R^d$. Given a vector field $F:\R^d \to \R^d$, $\diver F= \partial_i F^i$ is the divergence of $F$. Here and in all the paper the Einstein summation convention over repeated indices will be adopted. 

\medskip 

When we consider a function $w$ on the manifold $(\R^d,g)$, we use a different notation for its gradient. More precisely, we denote by $\nabla u = (\nabla^i u)_{i=1,\ldots,d}$ the Riemannian gradient of $u$ and recall that $\nabla u (x) =  g^{-1}(x) Du (x)$ for any $x \in \R^d$, where $g^{-1}=(g^{ij})_{i,j=1,\ldots,d}$ is the inverse matrix of $g$ given by \eqref{g_inverse}. 

It is readily seen that $\det{g}=\alpha^{-2}$ is constant: indeed the matrix $(g_{ij})_{i,j=1,\ldots,d}$  has $d-1$ eigenvalues equal to $1$ and one eigenvalue equals to $\alpha^{-2}$.
Since $\det(g)=\alpha^{-2}$ is constant, the volume element $dV_g$ is $\alpha^{-1} dx$.

We denote by $\diver_g(F)$ the divergence of a smooth vector field $F$ on $M$, that is, in
local coordinates
$$
\diver_g(F)=\frac{1}{\sqrt{|g|}}\partial_i(\sqrt{|g|}F^i)\, .
$$
As usual,  and $\Delta_g u = \nabla_i \nabla^i u$ is the Laplace-Beltrami
operator of $u$, where $\nabla_i$ is the covariant derivative, that is in local coordinates, 
\begin{align*}
\Delta_g u = \divg  \nabla u& =  \frac{1}{\sqrt{|g|}} \partial_i ( \sqrt{|g|} \nabla^i u ) =  \frac{1}{\sqrt{|g|}} \partial_i ( \sqrt{|g|} g^{ij} \partial_j  u )\, ,
\end{align*}
for all $u:\R^d\rightarrow\mathbb{R}$ smooth (say $C^2$). Notice that, since $\det(g)$ is constant, then
$$
\divg F = \diver F \ \ \text{ and } \ \ \Delta_g u =  \partial_i ( g^{ij} \partial_j  u ).
$$
Moreover, we write $g(\cdot,\cdot)$ for the scalar product induced by $g$, and we set 
$$ |\nabla u|_g^2 = g(\nabla u, \nabla u) = g^{ij} \, \pa_i u \, \pa_j u \,.$$  
The Hessian of $u$, when seen as a $(0,2)$-tensor, is defined as follows 
$$(\Hess_g u)_{ij} =  \partial_{ij} u - \Gamma^k_{\ ij} \ \partial_{k} u.$$ 
Given a $(0,2)$-tensor field $M$, 
we have that 
$\trace_g (M) = g^{ij}m_{ji}$ where  $M=(m_{ij})_{i,j=1,\ldots,d} $.

\medskip

\subsection{About the Riemmanian metric $g$}
Let $g$ be given by \eqref{g_def}. In terms of the coordinates of $\R^d$ 
the components of the Ricci tensor $\ric_g$ are given by 
\begin{align} \label{Ricci_jk}
R_{jk}=(1-\alpha^2)\frac{d-2}{|x|^2}\left (\delta_{jk}-\frac{x_jx_k}{|x|^2}\right ) 	\,,
\end{align}
and hence 
\begin{equation} \label{ricci}
\ric_g(\nabla v,\nabla v)=(1-\alpha^2)\frac{d-2}{|x|^2}\left (|\nabla v|^2-\frac{(\nabla v \cdot x)^2}{|x|^2}\right ) \,. 
\end{equation}
Notice that, by Cauchy-Schwarz inequality, $\ric_g(\nabla v,\nabla v) \geq 0$ when $\alpha^2 \leq 1$.

Since we are dealing with a weighted $p$-Laplace operator, with weight given by $|x|^{n-d}$, the gradient and the Hessian of the weight play a crucial role. Since 
$$
\frac{\partial^2}{\partial x_i \partial x_j}  |x|^{n-d} =(n-d)\frac{1}{|x|^2}(\delta_{ij}-2\frac{x_ix_j}{|x|^2})\,, 
$$
we have that 
\begin{equation} \label{hessiano}
\Hess_g(|x|^{n-d}) [\nabla v,\nabla v]=
\alpha^2\frac{n-d}{|x|^2}\left (|\nabla v|^2-\frac{(\nabla v\cdot x)^2}{|x|^2}\right )-(n-d) \frac{(\nabla v\cdot x)^2}{|x|^4} \,.
\end{equation}

\medskip

\subsection{The weighted $p$-Laplace operator}
As we have seen, we are dealing with a critical weighted $p$-Laplace type equation in the Riemannian manifold $(\R^d,g)$ given by \eqref{pb_w_Riem}. In order to lighten the notation, the weighted $p$-Laplace operator in the metric $g$ is sometimes denoted by
\begin{equation} \label{Lcal}
\mathcal{L} w :  = |x|^{-(n-d)}\diver (|x|^{n-d}|\nabla w|_g^{p-2}\nabla w) \,.
\end{equation}
The operator $\mathcal L$ is degenerate both for the presence of a $p$-Laplace type structure and for the presence of the weight that vanishes at the origin. 

A crucial quantity in our analysis is the \emph{stress} field $\mathcal A(\nabla v)$, which is defined by
\begin{equation} \label{mathcalA_def}
\mathcal A (\nabla w) = |\nabla w|_g^{p-2}\nabla w \,,
\end{equation}
and hence $\mathcal L w$ can be written as
$$
\mathcal{L} w :  = |x|^{-(n-d)}\diver (|x|^{n-d} \mathcal A (\nabla w) ) \,.
$$
Unless otherwise stated, $\mathcal A (\nabla w)$ is given by \eqref{mathcalA_def}. However, the operator $\mathcal L$ may be degenerate or singular at points where $\nabla w=0$ and then we need to argue by approximation and, for this reason, we are going to consider more general operators. A possible choice of the approximating stress field $\mathcal A_\ell $ is given by  the convolution of $\mathcal A$ with a family of radially symmetric smooth mollifiers $\{\phi_\ell\}$:
\begin{equation} \label{A_ell}
\mathcal A_\ell (\xi) : = (\mathcal A \ast \phi_\ell)(\xi) \,. 
\end{equation}
We will come back later on this approximation.

Forgetting for now the approximation problem, we clarify some notation and tools that we are going to use. We consider a function $V:T\R^d\to \R$ which is defined on the tangent bundle $T\R^d$. It is clear that $T\R^d= \R^d\times \R^d$ and we denote by $(x,\xi)$ a generic point in the tangent bundle, i.e. $x_1,\ldots,x_d$ are the coordinates on the manifold and $\xi_1,\ldots,\xi_d$ are the coordinates on the tangent space. Let $\mathcal{A} : \R^d\times \R^d \to \R^d \times \R^d$ be defined by
$$
\mathcal A^i (x,\xi) = \nabla_{g_x}^i V(x,\xi) := g_x^{ij} \partial_{\xi_j} V(x,\xi) \, \quad  i = 1,\ldots, d, \quad \text{for any } (x,\xi) \in \R^d \times \R^d \,. 
$$
In particular, if $B:[0,+\infty) \to \R$ is a smooth function and  
\begin{equation} \label{V_modxi}
V(x,\xi) = B(|\xi|_{g_x})\,,
\end{equation}
then
$$
\mathcal A (x,\xi) = \frac{B'(|\xi|_{g_x})}{|\xi|_{g_x}} \xi
$$
for any vector field $\xi$ in the tangent space at $x$. Notice that if $B(t)=t^p/p$ then $\mathcal{A} (x,\xi) = |\xi|_{g_x}^{p-2} \xi$ and $\diver \mathcal{A} (x,\nabla v) = \Delta_p v $\,.

In order to lighten the notation and whenever it does not create confusion, we omit the dependency on $x$ and we set
\begin{equation} \label{Axi}
\mathcal A(\xi) = \nabla V(\xi) 
\end{equation}
for any tangent vector $\xi$. It is clear that when we write \eqref{Axi} we mean that we are working on a fixed tangent space, so $x$ is fixed, and $\xi$ is a vector field belonging to the tangent space at $x$.  

We are going to evaluate $V$ and $\mathcal A$ at $\xi = \nabla v(x)$; in this case we have 
$$
\mathcal A(\nabla v(x)) = \nabla_\xi V(\nabla v(x)) \,,
$$
where the subscript $\xi$ emphasizes that the derivatives in $\nabla_\xi V$ are meant in the tangent space at $x$ (and not with respect to $x$, which is the case of $\nabla v$ without the subscript $\xi$). We also notice that if \eqref{V_modxi} holds then 
$$
\nabla_\xi^i \nabla_\xi^j V(\xi) = \left[ \frac{B''(|\xi|_{g_x})}{|\xi|_{g_x}^2} -  \frac{B'(|\xi|_{g_x})}{|\xi|_{g_x}^3} \right] \xi^i \xi^j + \frac{B'(|\xi|_{g_x})}{|\xi|_{g_x}} \delta^{ij} \,,
$$
and for $B(t)=t^p/p$ we have
$$
\nabla_\xi^i \nabla_\xi^j V(\xi) = (p-2) |\xi|_{g_x}^{p-4} \xi^i \xi^j + |\xi|_{g_x}^{p-2} \delta^{ij} \,.
$$

\section{A crucial differential identity} \label{sect_identity}
In this section we prove a differential identity which will be crucial in the rest of the paper.
We first prove a preliminary lemma.

\begin{lemma} \label{lemma_riemann}
Let $v: \mathbb{R}^d \to \mathbb{R}$ and $V: \R^d \times \mathbb{R}^d \to \mathbb{R}$ be smooth functions and set $\mathcal A(\nabla v(x)) = \nabla_\xi V(\nabla v(x))$. For any $x \in \R^d$ we have
\begin{multline} \label{crucial1}
 \nabla_j \left( \mathcal{A}^i (\nabla v(x))\nabla_i \mathcal{A}^j (\nabla v(x))\right)  =     \nabla_j \mathcal{A}^i (\nabla v(x))\nabla_i \mathcal{A}^j (\nabla v(x))+  \\ (\nabla_j\nabla_i \mathcal{A}^i(\nabla v(x)))\mathcal{A}^j (\nabla v(x))+  (\nabla_\xi^k \nabla_\xi^i V(\nabla v(x)) ) \nabla_\xi^j V(\nabla v(x)) R_{\ell kij}(x) \nabla^\ell v \,,
\end{multline}
where $R_{\ell kij}$ is the Riemann curvature tensor. In particular, if $V$ is as in \eqref{V_modxi} then we have that
\begin{multline} \label{crucial2}
 \nabla_j \left( \mathcal{A}^i (\nabla v(x))\nabla_i \mathcal{A}^j (\nabla v(x))\right)  =     \nabla_j \mathcal{A}^i (\nabla v(x))\nabla_i \mathcal{A}^j (\nabla v(x))+  \\ (\nabla_j\nabla_i \mathcal{A}^i(\nabla v(x)))\mathcal{A}^j (\nabla v(x))+ {\rm Ric}(\mathcal A (\nabla v), \mathcal A (\nabla v)) \,. \end{multline}
\end{lemma}

\begin{proof}
Before starting the proof, we make some remarks on the notation and the conventions we are going to use in this proof. Whenever it does not create confusion and aiming at lightening the notation, we omit the dependency of $\mathcal{A}(\nabla v(x))$ and $V(\nabla v(x))$ on $\nabla v(x)$. In order to make the calculation easier we consider a normal coordinate system centered at the point under consideration. In this setting, the commutation rule allows us to interchange the indexes in the third derivatives and we have
\begin{align}
	v_{ij}=v_{ji}, \qquad v_{ijk}-v_{ikj}=-R_{\ell ijk}v^\ell\\
	v_{kii}=v_{iki}=v_{iik}+{\rm Ric}_{ik}v^i,
	\end{align}
	where we used the convetions $v^\ell= \nabla^\ell v$,  $v_{ij}=\nabla_j \nabla_i v$, $v_{ijk}=\nabla_k \nabla_j \nabla_i v$, and $R $ and ${\rm Ric}$ are the Riemann curvature and Ricci tensors, respectively. 	
	
The starting point of this proof is the following identity:
\begin{equation} \label{888_proof}
(\nabla_j\nabla_i \mathcal{A}^i)\mathcal{A}^j -(\nabla_j\nabla_i \mathcal{A}^j)\mathcal{A}^i =
  - (\nabla_\xi^k \nabla_\xi^i V ) \nabla_\xi^j V R_{\ell kij} v^\ell \,.
\end{equation}
In order to prove \eqref{888_proof} we notice that, since $A^j = \nabla^j_\xi V$, we have
\begin{equation*}
\begin{split}
\nabla_j\nabla_i \mathcal{A}^j (\nabla v) & = \nabla_j\nabla_i \nabla^j_\xi V(\nabla v) \\
& = \nabla_j ( \nabla_\xi^k \nabla_\xi^j V (\nabla v) v_{ki}) \\
& = (\nabla_\xi^\ell \nabla_\xi^k \nabla_\xi^j V (\nabla v)) v_{ki} v_{\ell j} +  (\nabla_\xi^k \nabla_\xi^j  V (\nabla v)) v_{kij} 
\end{split}
\end{equation*}
and
\begin{equation*}
\begin{split}
\nabla_j\nabla_i \mathcal{A}^i  (\nabla v) & = \nabla_j\nabla_i \nabla_{\xi}^i V (\nabla v)  \\
& = \nabla_j ( \nabla_{\xi}^k \nabla_{\xi}^iV (\nabla v) v_{ki}) \\
& = (\nabla_{\xi}^\ell \nabla_{\xi}^ k \nabla_{\xi}^iV (\nabla v)) v_{ki} v_{\ell j} +  (\nabla_{\xi}^k \nabla_{\xi}^i V (\nabla v)) v_{kij} \,,
\end{split}
\end{equation*}
and hence
\begin{equation*}
\begin{split}
(\nabla_j\nabla_i \mathcal{A}^i)\mathcal{A}^j -(\nabla_j\nabla_i \mathcal{A}^j)\mathcal{A}^i& = (\nabla_{\xi}^\ell \nabla_{\xi}^ k \nabla_{\xi}^iV ) \, (\nabla_{\xi}^j V) v_{ki} v_{\ell j} +  (\nabla_{\xi}^k \nabla_{\xi}^i V) \, (\nabla_{\xi}^j V) v_{kij} \\ 
& \quad -  (\nabla_\xi^\ell \nabla_\xi^k \nabla_\xi^j V)\, (\nabla_\xi^i V) v_{ki} v_{\ell j} -  (\nabla_\xi^k \nabla_\xi^j)  V \, (\nabla_\xi^i V )v_{kij}  \\
& =   (\nabla_\xi^k \nabla_\xi^i V)   \, (\nabla_\xi^j V)  v_{kij}  -  (\nabla_\xi^k \nabla_\xi^j V) \, (\nabla_\xi^i V) v_{kij}  \\
& =  (\nabla_\xi^k \nabla_\xi^i V)   \, (\nabla_\xi^j V)  v_{kij}  -  (\nabla_\xi^k \nabla_\xi^i V) \, (\nabla_\xi^j V) v_{kji} \,,
\end{split}
\end{equation*}
where in the last two equalities we simplified two terms and rearranged the indices.  
%
Since $v_{kij}-v_{kji}=-R_{\ell kij}v^\ell$, we have \eqref{888_proof}. From 
\begin{equation*} 
 \nabla_j \left( \mathcal{A}^i \nabla_i \mathcal{A}^j \right) =    \nabla_j \mathcal{A}^i \nabla_i \mathcal{A}^j + \mathcal{A}^i \nabla_j\nabla_i\mathcal{A}^j  \,.
\end{equation*}
and by using \eqref{888_proof}, we obtain \eqref{crucial1}. Identity \eqref{crucial2} immediately follows from \eqref{crucial1}.
\end{proof}

The following proposition contains the crucial identity that we are going to use in the proof of Theorems \ref{th_main_p} and \ref{main_thm_Riem} (see \cite{BiCi,BCS} for an analogous identity in the Euclidean case without weight). We prove such identity in a more general context, which may be useful for future investigations. More precisely, we consider the following operator
\begin{equation} \label{Lf_general}
\mathcal{L}_f v  = e^{f} \diver (e^{-f} \mathcal{A}(\nabla v) ) \,,
\end{equation}
where the weight $e^{-f}$ and the field $\mathcal{A}$ are smooth. It is clear that if $e^{-f}=|x|^{n-d}$ and $\mathcal{A}(\xi)=|\xi|_g^{p-2}\xi$ then $\mathcal{L}_f = \mathcal L$ is the weighted $p$-Laplace operator defined in \eqref{Lcal} (apart from the smoothness issue at the origin).

\begin{proposition} \label{prop_diff_id}
Let $f,v: \mathbb{R}^d \to \mathbb{R}$ and $V: \R^d \times \mathbb{R}^d \to \mathbb{R}$ be smooth functions. Let $\mathcal A(\xi) = \nabla V(\xi)$ as in \eqref{Axi}, and let $\mathcal{L}_f$ be given by \eqref{Lf_general}. For any $x \in \R^d$ we set $W(x)=(W_j^i(x))_{i,j=1,\ldots,d}$, where
$$
W_j^i (x) = \nabla_j \mathcal A^i (\nabla v(x)) \,, \quad i,j=1,\ldots,d.
$$ 
We have the following differential identity 
	\begin{multline} \label{ident_final_generale}
	e^f \diver \left( -e^{-f} v^\gamma \mathcal{A}^i(\nabla v) \nabla_i \mathcal{A}^j(\nabla v) + v^\gamma  e^{-f} \mathcal{A}^j(\nabla v) \mathcal{L}_fv  + \gamma (p-1) v^{\gamma-1} V(\nabla v)\, e^{-f}\mathcal{A}^j(\nabla v) \right) \\
	= v^{\gamma} \left[ (\mathcal{L}_f v)^2 - \trace (W^2) \right] + \gamma v^{\gamma-1}[(p-1) V(\nabla v) + \mathcal{A}^j(\nabla v) \nabla_j v] \mathcal{L}_f v \\ + \gamma(\gamma-1) (p-1) v^{\gamma-2} V(\nabla v) \mathcal A^j(\nabla v) \nabla_j v 
	\\ + \gamma v^{\gamma-1} \mathcal A^i(\nabla v) \nabla_i (p V(\nabla v)  - \mathcal A^j (\nabla v) \nabla_j v) \\ - \Hess f ( \mathcal A(\nabla v),\mathcal A(\nabla v)) - \nabla_\xi^k\nabla_\xi^i V (\nabla v) \mathcal A^j(\nabla v) R_{\ell k i j } \nabla^\ell v \,. 	
	\end{multline}
In particular, if $V(\nabla v) = \frac{1}{p} |\nabla v|_g^p$ then \eqref{ident_final_generale} is
		\begin{multline} \label{ident_final}
	e^f \diver \left( -e^{-f} v^\gamma \mathcal{A}^i(\nabla v) \nabla_i \mathcal{A}^j(\nabla v) + v^\gamma  e^{-f} \mathcal{A}^j(\nabla v) \mathcal{L}v  + \gamma\frac{p-1}{p} v^{\gamma-1} |\nabla v|_g^p\, e^{-f}\mathcal{A}^j(\nabla v) \right) \\
	= v^{\gamma} \left[ (\mathcal{L} v)^2 - \trace (W^2) \right] + \gamma \frac{2p-1}{p} v^{\gamma-1} |\nabla v|_g^p \mathcal{L}v + \gamma(\gamma-1)\frac{p-1}{p} v^{\gamma-2} |\nabla v|_g^{2p} \\
	- v^{\gamma} |\nabla v|_g^{2(p-2)}  {\rm Ric}(\nabla v,\nabla v) - v^{\gamma} |\nabla v|_g^{2(p-2)}\Hess f  (\nabla v,\nabla v) \,.
	\end{multline}
	\end{proposition}

\begin{proof}
As in the proof of Lemma \ref{lemma_riemann} we omit the dependency of $\mathcal A$ on $\nabla v$. From Lemma \ref{lemma_riemann} we have that	
\begin{multline} \label{id1}
	e^f \nabla_j \left( -e^{-f} v^\gamma \mathcal{A}^i \nabla_i \mathcal{A}^j \right) = 
	v^\gamma  \mathcal{A}^i f_j \nabla_i  \mathcal{A}^j - \gamma v^{\gamma-1} v_j \mathcal{A}^i\nabla_i\mathcal{A}^j  - v^{\gamma} \nabla_j \mathcal{A}^i \nabla_i \mathcal{A}^j  \\ - v^\gamma \left(    \nabla_j \mathcal{A}^i (\nabla v(x))\nabla_i \mathcal{A}^j (\nabla v(x))+  (\nabla_j\nabla_i \mathcal{A}^i(\nabla v(x)))\mathcal{A}^j (\nabla v(x))+ {\rm Ric}(\mathcal A (\nabla v), \mathcal A (\nabla v)) \right)  \,.
	\end{multline}
A straightforward computation yields
	\begin{equation} \label{id2}
	\begin{split}
	e^f \nabla_j \left( v^\gamma e^{-f} \mathcal{A}^j \mathcal L_f v \right) & = \gamma v^{\gamma-1}\mathcal{A}^j v_j \mathcal L_f v + v^\gamma  \nabla_j(\mathcal L_fv) \mathcal{A}^j + v^\gamma (\mathcal L_f v)^2 \\ 
	& = \gamma v^{\gamma-1} \mathcal{A}^j v_j \mathcal L_fv + v^\gamma \mathcal{A}^j \nabla_j\nabla_i\mathcal{A}^i  - v^\gamma \nabla_j( f_i \mathcal{A}^i) \mathcal{A}^j + v^\gamma (\mathcal L_f v)^2\,,
	\end{split}
	\end{equation}
	and
	\begin{multline} \label{id3}
	e^f \nabla_j \left( \gamma (p-1) v^{\gamma-1} V \, e^{-f} \mathcal{A}^j \right) 
	= \gamma (\gamma-1)(p-1) v^{\gamma-2} V  \mathcal{A}^j v_j \\ + \gamma(p-1) v^{\gamma-1}  \mathcal{A}^j  \nabla_j V  + \gamma (p-1) v^{\gamma-1} V  \mathcal L_f v \,,
	\end{multline}
where in \eqref{id2} we have used that $\mathcal L_f v = \nabla_i \mathcal A^i - f_i \mathcal A^i $.

Now we notice that
	\begin{equation*} 	
	\begin{split}
	\nonumber (p-1) \nabla_j V  \mathcal{A}^j-(\nabla_i \mathcal{A}^j)\mathcal{A}^iv_j & =(p-1) \nabla_j V  \mathcal{A}^j -\nabla_i (\mathcal{A}^jv_j)\mathcal{A}^i+\mathcal{A}^j\mathcal{A}^iv_{ji}	\\
	& = p \nabla_j V  \mathcal{A}^j-\mathcal A^i \mathcal{A}^j v_{ji}-\nabla_i (\mathcal{A}^jv_j)\mathcal{A}^i+\mathcal{A}^j\mathcal{A}^iv_{ji}
	\end{split}
	\end{equation*}
and hence
\begin{equation}\label{3}
(p-1) \nabla_j V  \mathcal{A}^j-(\nabla_i \mathcal{A}^j)\mathcal{A}^iv_j = \nabla_i (p V - \mathcal{A}^jv_j)\mathcal{A}^i \,.
\end{equation}

	Moreover, we have that 
	\begin{equation} \label{4}
	\mathcal{A}^i f_j \nabla_i \mathcal{A}^j - \nabla_j (f_i \mathcal{A}^i) \mathcal{A}^j = - f_{ij} \mathcal{A}^i \mathcal{A}^j\,.
	\end{equation}
	From \eqref{id1}-\eqref{4} we obtain \eqref{ident_final_generale}. If $V(\nabla v) = \frac{1}{p} |\nabla v|_g^p$ then $\mathcal A (\nabla v) = |\nabla v|_g^{p-2} \nabla v$ and \eqref{ident_final} immediately follows.  
\end{proof}

As already mentioned, \eqref{ident_final} is the crucial tool in the proof of the main theorem. Due to the lack of regularity, we cannot apply \eqref{ident_final} directly and we need to provide an integral version of 	\eqref{ident_final} which is valid for solutions to \eqref{EL_w_p_II}. In order to do that, we need several regularity results which are proved in the next two sections.

\section{Preliminary regularity results and asymptotic estimates} \label{sect_reg1}
In this section we recall some known results on the regularity of the solution and we will prove some preliminary result on regularity and asymptotic estimates at infinity and at the origin. Regularity results for the field $|\nabla w|_g^{p-2}\nabla w$ will be proved in Section \ref{section_reg_A}.

Coming back to this section, Subsection \ref{subsec_bounded} is devoted to prove the boundedness of solutions in a slightly more general setting, which is needed for the approximation argument used in Section \ref{section_reg_A}. In Subsection \ref{subsec_asympt} we recall some asymptotic estimates at infinity obtained in \cite{Vetois19} and give asymptotic estimates for the second derivatives of solutions to \eqref{eq_main_u_p}.

\subsection{Boundedness of solutions.} \label{subsec_bounded} In the following result we prove that solutions of problems involving approximations of $p$-Laplacian operators are locally bounded. This result is a generalization of results in \cite{ColoradoPeral,Peral,Serrin_local} to the present setting. It is important to emphasize that the $L^\infty$ estimate is uniform with respect to the approximation that we use.

\begin{proposition}
	\label{bounded_sol}
	Let $\Omega \subseteq \R^d$, with $d\geq 3$, and let $w$ be a solution of
	\begin{equation}
	\label{eq_f}
	\begin{cases}
	\dfrac{1}{|x|^{n-d}}\diver(|x|^{n-d}\mathcal{A}(Dw))=\Phi(x,w) \;\text{ in }\R^d,\\
	w>0,\\
	w\in \DD^{1,p}(\Omega,|x|^{n-d}),
	\end{cases}
	\end{equation}
	where $\mathcal{A}:\mathbb{R}^d\rightarrow\mathbb{R}^d$ is a continuous vector field 
	such that 
	there exist $\gamma>0$ and $0\leq s\leq 1/2$ such that
	\begin{equation}\label{ellipticity_approx}
	|\mathcal{A}(\xi)|\leq \gamma(|\xi|^2+s^2)^{\frac{p-1}{2}} \quad \text{and} \quad 	\xi\cdot \mathcal{A}(\xi)\geq \dfrac{1}{\gamma}\int_0^1\left( t^2|\xi|^2+s^2\right)^{\frac{p-2}{2}}|\xi|^2\, dt  \, , 
	\end{equation}
	for every $\xi\in\mathbb{R}^d$, 	$f:\R^d \times (0,\infty) \to \R$ is a Caratheodory function satisfying 
	\begin{equation}
	\label{bound_f}
	|\Phi(x,z)|\leq \Lambda z^{q-1}+ \Psi(x)\qquad \text{for all $z\in \R$},
	\end{equation}
	for some constant $\Lambda>0$, $p>1$, $a\leq b\leq a+1$, $a\leq a_c$, $q=\frac{dp}{d-p(1+a-b)}$ and some function $\Psi \in L^\infty_{loc}(\R^d)$. 
	Then $\|w\|_\infty \leq C$, where $C$ does not depend on $s$. 
\end{proposition}

\begin{proof}
	We follow the proofs of \cite[Theorem E.0.20]{Peral} and \cite[Theorem 1]{Serrin_local} (see also \cite[Lemma 2.1]{CFR} and \cite[Lemma 2.1]{CCR}). We first notice that, as in \cite[Lemma 2.1]{CCR}, there exists $\gamma_*>0$ such that  
	\begin{equation}\label{eqclaim}
	\xi\cdot \mathcal{A}(\xi)\geq \gamma_* \left(|\xi|^p - s^p\right) \,.
	\end{equation}
	for every $\xi\in\mathbb{R}^d$.
	
	Let $\tilde{w}=w+s$. By \eqref{ellipticity_approx} we obtain that $\tilde w$ satisfies 
	\begin{equation} \label{starting}
	|\mathcal{A}(\nabla \tilde w)|_g\leq \gamma_*(|\nabla \tilde w|_g^2+\tilde{w}^2)^{\frac{p-1}{2}} \quad \text{and} \quad g(\nabla \tilde w, \mathcal{A}(\nabla \tilde w))\geq \dfrac{1}{2\gamma_*}\left(|\nabla \tilde w|_g^p - \tilde{w}^p\right), 
	\end{equation}
	which are our starting point. In order to avoid heavy notation, we write $w$ instead of $\tilde w$.

	Given $l>0$ and $t \geq 1$, we define
	\begin{equation*}
	F(w)=\begin{cases}
	w^{t}\ & \text{ if } w\leq l \\
	tl^{t-1}(w-l)+l^t & \text{ if } w>l  \, ,
	\end{cases}
	\end{equation*}
	and 
	$$
	G(w)=\begin{cases}
	w^{(t-1)p+1}\ & \text{ if } w\leq l \\
	((t-1)p+1)l^{(t-1)p}(w-l)+ l^{(t-1)p+1} & \text{ if } w>l  \,.
	\end{cases}
	$$
	Let 
	$$
	\xi=\eta^p G(w)
	$$
	where $\eta\in C^\infty_c(\mathbb{R}^d)$ and $\eta\geq 0$. From \eqref{eq_f} with $\xi$ used as test-function, we obtain
	\begin{equation}\label{debole}
	\int_{\R^d} |x|^{n-d}g(\mathcal{A}(\nabla w), \nabla(\eta^p G(w)))\, dx=\int_{\R^d}|x|^{n-d}\Phi(x,w)\eta^p G(w)\, dx\, .
	\end{equation}
	From \eqref{debole}, \eqref{starting} and by the fact that $G\geq 0$ we get 
	\begin{align*}
	c_1\int_{\R^d}|x|^{n-d} \eta^p G'(w)|\nabla w|_g^p\, dx &\leq  \int_{\R^d}|x|^{n-d} \eta^{p-1}G(w)|g(\mathcal{A}(\nabla w), \nabla\eta)|\, dx+\int_{\R^d} |x|^{n-d} w^p\eta^p G'(w)\, dx \, \\&+\int_{\R^d}|x|^{n-d}\Phi(x,w)\eta^p G(w)\, dx 
	,
	\end{align*}
	for some $c_1>0$. We estimate the second term by using Young's inequality and  \eqref{ellipticity_approx}, and we obtain
	\begin{equation*}
	\begin{aligned}
	\eta^{p-1} |g(\mathcal{A}(\nabla w), \nabla \eta)|&\leq
	\epsilon^{\frac{p}{p-1}}w^{-1}|\mathcal{A}(\nabla w)|_g^{\frac{p}{p-1}}\eta^p+\epsilon^{-p}w^{p-1}|\nabla\eta|_g^p\\
	&\leq C_1\epsilon^{\frac{p}{p-1}}w^{-1}(|\nabla w|_g^p+w^p)\eta^p+\epsilon^{-p}w^{p-1}|\nabla\eta|_g^p,
	\end{aligned}
	\end{equation*}
	for any $\epsilon\in (0,1)$, where $C_1$ depends only on $\gamma$ and $p$.
	From \eqref{bound_f}, since $G(w)\leq wG'(w)$ and $G$ is convex, we obtain 
	\begin{equation*}
	\begin{aligned}
	c_1\int_{\R^d} |x|^{n-d}\eta^p G'(w)|\nabla w|_g^p\, dx &\leq C_1\epsilon^{\frac{p}{p-1}} \int_{\R^d} |x|^{n-d}\eta^p G'(w)|\nabla w|_g^p\, dx
	+   C_2 \int_{\R^d} |x|^{n-d} w^p \eta^p G'(w)\, dx \\
	&+\epsilon^{-p}\int_{\R^d}|x|^{n-d} G(w)w^{p-1}|\nabla \eta|_g^p\, dx+ \Lambda\int_{\R^d}  |x|^{n-d} w^{q-1}\eta^p G(w)\, dx\, ,\\
	&+  \int_{\R^d}  |x|^{n-d}\eta^p \Psi(x)  G(w)\, dx\,,
	\end{aligned}
	\end{equation*}
	for any $\epsilon\in (0,1)$ and for some constant $C_2$ which depends only on $\gamma$ and $p$. We choose $\epsilon$ small enough and obtain
	\begin{align*}
	c_2\int_{\R^d} |x|^{n-d}\eta^p G'(w)|\nabla w|_g^p\, dx &\leq \int_{\R^d}  |x|^{n-d}\eta^p w^p G'(w)\, dx 
	+\int_{\R^d} |x|^{n-d} G(w)w^{p-1}|\nabla \eta|_g^p\, dx\\
	&+ \int_{\R^d}  |x|^{n-d}\eta^p w^{q-1} G(w)\, dx\,+\int_{\R^d}  |x|^{n-d}\eta^p \Psi(x) G(w)\, dx\, ,
	\end{align*}
	where $c_2>0$ depends only on $n$, $p$, $\gamma$ and $\Lambda$.
	Since 
	$
	G'(w)\geq c[F'(w)]^p
	$
	, 
	$
	w^{p-1}G(w)\leq C[F(w)]^p \,,
	$ and $\Psi \in L^{\infty}_{loc}(\R^d)$ 
	we obtain  
	\begin{equation}
	\label{serrin_general}
	\begin{aligned}	
	c_3\int_{\R^d} |x|^{n-d} |\nabla (\eta F(w))|_g^p\, dx & \leq \int_{\R^d}|x|^{n-d}\eta^pw^p G'(w)\, dx+\int_{\R^d}|x|^{n-d} |\nabla \eta|_g^pF^p(w)\, dx\\
	&+ \int_{\R^d}|x|^{n-d}\eta^p  w^{q-p} F^p(w)\, dx\, + \int_{\R^d}|x|^{n-d}\eta^p   w^{1-p} F^p(w)\, dx\,,
	\end{aligned}
	\end{equation}
	where $c_3$ depends only on $n$, $p$ and  $\gamma$. 	Thanks to Caffarelli-Kohn-Nirenberg inequality \eqref{CKN}, we find
	\begin{equation}\label{post_sobolev}
	\begin{aligned}
	c_4\left(\int_{\R^d} |x|^{n-d}F^{q}(w)\eta^{q}\, dx\right)^{\frac{p}{q}} & \leq \int_{\R^d}|x|^{n-d}\eta^pw^pG'(w)\, dx+\int_{\R^d}|x|^{n-d}|\nabla \eta|_g^pF^p(w)\, dx \\&+ \int_{\R^d}|x|^{n-d}\eta^pw^{q-p} F^p(w)\, dx \,+ \int_{\R^d}|x|^{n-d}\eta^p w^{1-p} F^p(w)\, dx ,
	\end{aligned}
	\end{equation}
	where $c_4>0$ depends only on $n$, $p$, $\gamma$ and the Sobolev constant for $\R^d$.\\		
	Let $\rho>0$ be such that 
	$$
	\left (\int_{B_\rho(x_0)}|x|^{n-d}w^{q} \, dx \right )^{\frac{q-p}{q}} \leq \frac{c_4}2.
	$$
	Let $R < \rho$ and let $\eta$ be such that ${\rm supp}(\eta)\subset B_R(x_0)$.  From Holder's inequality applied to the last term in \eqref{post_sobolev},  we obtain
	\begin{align*}
	\frac{c_4}2\left(\int_{\R^d}|x|^{n-d}F^{q}(w)\eta^{q}\, dx\right)^{\frac{p}{q}}\leq& \int_{\R^d}|x|^{n-d}\eta^pw^{p}G'(w)\, dx+\int_{\R^d}|x|^{n-d}|\nabla \eta|_g^pF^p(w)\, dx\, \\
	+&\int_{\R^d}|x|^{n-d} \eta^pw^{1-p}F^p(w)\, dx\, .
	\end{align*}
	By taking the limit as $l\rightarrow\infty$, from the definition of $F$ and $G$ and since $\eta \geq 0$, by monotone convergence we conclude that 
	\begin{equation*}
	\begin{aligned}
	\frac{c_4}2\left(\int_{\R^d}|x|^{n-d} \eta^{q}w^{tq}\, dx\right)^{\frac{p}{q}}&\leq
	\int_{\R^d}|x|^{n-d}\eta^pw^{p}w^{(t-1)p}\, dx + \int_{\R^d}|x|^{n-d}|\nabla\eta|_g^p w^{tp}\, dx \\ 
	&+ \int_{\R^d}|x|^{n-d}\eta^p w^{1-p}w^{tp}\, dx\,\\
	&\leq 
	\int_{\R^d}|x|^{n-d}(\eta^p+|\nabla \eta|_g^p)w^{tp}\, dx+ \int_{\R^d}|x|^{n-d}\eta^p  w^{1+(t-1)p} \, dx .
	\end{aligned}
	\end{equation*}
	Moreover, since $1+(t-1)p<tp$, by Holder inequality we have 
	\begin{equation*}
	\begin{aligned}
	\frac{c_4}2\left(\int_{\R^d}|x|^{n-d} \eta^{q}w^{tq}\, dx\right)^{\frac{p}{q}}\leq 
	\int_{\R^d}|x|^{n-d}(\eta^p+|\nabla \eta|_g^p)w^{tp}\, dx + \left (\int_{\R^d}|x|^{n-d}\eta^p  w^{tp} \, dx \right)^{\frac{1+(t-1)p}{tp}},
	\end{aligned}
	\end{equation*}
	Hence, if $\rho>R>R'>0$ and we take $\eta\in C_c^\infty(B_R(x_0))$, $0\leq \eta \leq 1$, $\eta=1$ in $B_{R'}(x_0)$, $|\nabla\eta|\leq \frac{1}{R-R'}$, then we have
	\begin{equation}
	\label{serrin0}
	\begin{aligned}
	\left(\int_{B_{R'}(x_0)} |x|^{n-d} w^{tq}\, dx\right)^{\frac{p}{q}}
	\leq c_5 \left (1+\frac{1}{R-R'} \right)  \left (\int_{B_{R}(x_0)}|x|^{n-d} w^{tp}\, dx \right )\\
	+ \left(\int_{B_{R}(x_0)}|x|^{n-d}w^{tp}\, dx\right )^{\frac{1+(t-1)p}{tp}}
	\end{aligned}
	\end{equation} 
	where $c_5>0$. 	Inequality \eqref{serrin0} is the starting point of a Moser iteration. Let $R_k=r(1+2^{-k})$  with $0<r<\rho/2$ and $t_j=\left (\frac{q}{p}\right )^k$ for $k=1,2,\dots$, and $t=t_k$. By setting $R=R_k$, $R'=R_{k+1}$ and $t=t_k$, a finite iteration of \eqref{serrin0} gives that 
$$
\int_{B_{R_k}(x_0)}  |x|^{n-d} w^{t_kq}\, dx<\infty
$$ for every $k>0$. Now, we prove by contradiction that 
$$
\displaystyle \sup_{k}\int_{B_{r}(x_0)}|x|^{n-d}w^{t_kp}\, dx <\infty \,.
$$ 
Indeed, by contradiction let us assume that 
$$
\displaystyle \sup_{k}\int_{B_{r}(x_0)}|x|^{n-d}w^{t_kp}\, dx =\infty\,,
$$
i.e. 
$$
\displaystyle \lim_{k\to +\infty} \int_{B_{R_k}(x_0)}|x|^{n-d}w^{t_kp}\, dx =\infty
$$ 
since 
$$
\displaystyle \int_{B_{r}(x_0)}|x|^{n-d}w^{t_kp}\, dx
$$ 
is non-decreasing in $k$. Then there exists $\widetilde k$ such that 
$$
\displaystyle \int_{B_{R_k}(x_0)}|x|^{n-d} w^{t_kp}\, dx >1
$$ 
for every $k\geq \widetilde k$; hence by \eqref{serrin0} and by the fact that $\frac{1+(t_k-1)p}{t_kp}<1$ we have for some $c_{6}>0$
	\begin{equation}
	\label{serrin1}
	\begin{aligned}
	\left(\int_{B_{R_{k+1}}(x_0)} |x|^{n-d} w^{t_{k+1}q}\, dx\right)^{\frac{1}{t_{k+1}q}}
	\leq c_{6} \left (1+\frac{1}{R_k-R_{k+1}} \right)^{\frac{1}{t_k}}  \left (\int_{B_{R_k}(x_0)} |x|^{n-d} w^{t_kp}\, dx \right )^{\frac{1}{t_kp}}
	\end{aligned}
	\end{equation} 
	for every $k\geq \widetilde k_0$. Iterating we can write 
	\begin{equation}
	\left(\int_{B_{R_k}(x_0)} |x|^{n-d}w^{t_kq}\, dx\right)^{\frac{1}{t_kq}}
	\leq c_{6}^{\sum_{j=k_0}^{k} \frac {1} {t_j}}\Pi_{j=k_0}^{k} \left (1+2^{j+1}r^{-1} \right)^{\frac {1} {t_j}}\left (\int_{B_{\rho}(x_0)} |x|^{n-d}w^{t_{k_0}p}\, dx \right )^{\frac{1}{t_{k_0}p}},
	\end{equation} 
	for $k>k_0$. Therefore 
$$
\displaystyle \sup_{k}\int_{B_{r}(x_0)}|x|^{n-d} w^{t_kp}\, dx \leq \displaystyle \sup_{k}\int_{B_{R_k}(x_0)}|x|^{n-d} w^{t_kp}\, dx <\infty \,.
$$  
The latter implies that  
	\begin{equation}
	\label{w_infty}
	\|w\|_{\infty, r}\leq \lim_{k\to \infty} \left(\int_{B_{r}(x_0)} |x|^{n-d} w^{t_kq}\, dx\right)^{\frac{1}{t_kq}}<\infty.
	\end{equation}
	Indeed, 
	$$
	K=K \lim_{k\to \infty}\left (\int_{A_k} |x|^{n-d} \right )^\frac{1}{t_kq}\leq \lim_{k\to \infty}\left (\int_{A_k} |x|^{n-d} w^{t_kq}\right )^\frac{1}{t_kq}<\infty.
	$$
	for every $K>0$, where $A_K=\{w>K\}\cap B_r(x_0)$. By definition of $\|w\|_{\infty, r}$, \eqref{w_infty} holds.   The proof is now completed recalling the substitution $\widetilde w=w+s$.
	
\end{proof}

\subsection{Asymptotic bounds at infinity} \label{subsec_asympt}
As it is shown in \cite{Vetois19}, the fact that $u\in \DD^{1,p}(\R^d,|x|^{-ap})$ implies that the asymptotic behavior of the solution at infinity can be optimally determined. 

\begin{proposition}\label{stime_Vetois}
	Let $d\geq 3$, $p>1$, $a\leq b\leq a+1$, $q=\frac{dp}{d-p(1+a-b)}$ and $a\leq a_c$ and let $w$ be a solution to \eqref{pb_w_Riem}.  Then there exists $C>0$ such that 
	\begin{equation}
	\label{stima_w_vetois}
	||x|^{\frac\mu\alpha } w(x)-\alpha|\leq |x|^{-\frac \delta\alpha},
	\end{equation}
	\begin{equation}
	\label{stime_grad_w}
	\frac 1 C \leq |x|^{\frac\mu\alpha+1} |\nabla  w(x)|_g\leq C,
	\end{equation}
	and
	\begin{equation}
	\label{stime_der_sec}
	|x|^{\frac\mu\alpha+2} |\Hess_g w(x)|\leq C 
	\end{equation}
	for every $x\in \R^d\backslash B_{1}$ and $\mu=\frac{d-p(1+a)}{p-1}$. 
\end{proposition}
\begin{proof}
	Estimates \eqref{stima_w_vetois} and \eqref{stime_grad_w} are straightforward consequences of \cite[Theorem 1.1]{Vetois19}. Indeed, we have that the solution $u$ to \eqref{eq_main_u_p} satisfies the following asymptotic estimates 
	\begin{equation*}
	||x|^\mu u(x)-\alpha|\leq C |x|^{-\delta} \qquad \forall x\in \R^d\backslash B_1
	\end{equation*}
	and 
	\begin{equation*}
	||x|^{\mu+1} D u(x)|\leq C \qquad \forall x\in \R^d\backslash B_1.
	\end{equation*}
	\begin{equation*}
	|x|^{\mu+1} D u(x)+\alpha \mu |x|^{-1}x \to 0 \text{ as } |x|\to \infty,
	\end{equation*}
	where $\mu=\frac{d-p(1+a)}{p-1}$, for some constants $\alpha,\delta,C>0$.
	From the definition of $w$ and since $g$ is zero-homogeneous we immediately have the assertions.	\\
	To prove \eqref{stime_der_sec} we use a scaling argument.  
	Let $\rho>1$ be fixed. For $y\in E=\overline{B_1}\backslash B_{1/2}$, we define 
	$$
	\widetilde{w}(y)=\rho^{\mu}w(\rho y).
	$$
	Since \eqref{pb_w_Riem} holds, $\widetilde{w}$ satisfies 
	\begin{equation}
	\diver(|y|^{n-d}|\nabla \widetilde w|^{p-2}\nabla \widetilde w)=-\rho^{(\frac\mu\alpha+1)(p-1)+1-\frac \mu \alpha (q-1)}|y|^{n-d}{\widetilde w}^{q-1} \;\text{ in }E.
	\end{equation}
	Thanks to \eqref{stime_grad_w} and by using the zero-homogeneity of $g$, from elliptic regularity and Schauder's estimates (see \cite{GT}) 
	we have that $|D^2 \widetilde w|\leq C$ for some $C>0$ and then \eqref{stime_der_sec}. 
\end{proof}

\section{Regularity estimates for $\mathcal{A}(\nabla w)$} \label{section_reg_A}
In this section, by using a weighted Caccioppoli inequality, we prove that $\mathcal{A}(\nabla w)\in W^{1,2}_{\rm loc}(\R^d,|x|^{n-d})$. Before proving this result, we give a preliminary asymptotic estimate on $\nabla w$ at the origin, which will be used later.

\begin{lemma} \label{lemma_gradient_x}
Let $w$ be a solution to \eqref{pb_w_Riem}. Then there exists a constant $C$ such that
$$
|\nabla w(x)|_g \leq \frac{C}{|x|} \,,
$$
for $x \in B_1 \setminus \{O\}$.
\end{lemma}

\begin{proof}
This proof follows by a scaling argument. Let $E:= B_4 \setminus \overline B_1$. For any $\mu \in (0,1/4)$ we define 
$$
\zeta_\mu (x) := w(\mu x) \,,
$$ 
for any $x \in E$. Since the metric $g$ zero homogeneous, we readily find that $\zeta_\mu$ satisfies
\begin{equation} \label{eq_L_zeta}
\mathcal{L} \zeta_\mu = - \mu^p \zeta_\mu^{q-1} 
\end{equation}
in $E$. We notice that, since $|x|>1$ in $E$, we have that the operator $\mathcal L$ does not degenerate in the variable $x$ and $\mathcal L$ is in the class of operators considered in \cite{DiBenedetto, Tolksdorf} (see also \cite{AntoniniCiraoloFarina}). From Proposition \ref{bounded_sol}, we have that $\|\zeta_\mu\|_{L^\infty(E)} \leq c$ uniformly with respect to $\mu$ and then the RHS in \eqref{eq_L_zeta} is uniformly bounded. From elliptic regularity theory \cite{DiBenedetto,Tolksdorf}, this implies that there exists a constant $C$ independent of $\mu$ such that 
$$
| \nabla \zeta_\mu(x) | \leq C 
$$
for any $x \in B_3 \setminus \overline B_2$.
By recalling the definition of $\zeta$, the assertion immediately follows.
\end{proof}

Since \eqref{pb_w_Riem} may be degenerate, in the following we argue by approximation. We recall that, for fixed $x \in \R^d$, $\mathcal A$ can be seen as the gradient of $V(\xi)= p^{-1} |\xi|_{g_x}^p$ with respect to $\xi$ which is in the tangent space of $\R^d$ at $x$, i.e. 
$$
\mathcal A^i (\xi) = \nabla_\xi^i V(\xi) \,, 
$$
for any $\xi$ in the tangent space at $x$.
Let $\{\phi_\ell\}_{\ell \in \mathbb{N}}$ be a family of radially symmetric smooth mollifiers and define
$$
V_\ell (\xi) = (V\ast\phi_\ell)( \xi )
$$
and then
	\begin{equation} \label{ajz}
	\mathcal{A}_\ell (\xi):=(\mathcal{A}\ast\phi_\ell)(\xi) \,.
	\end{equation}
Standard properties of convolution and the fact $\mathcal{A}(\cdot)$ is continuous imply that $\mathcal{A}_\ell \rightarrow \mathcal{A}$ uniformly on compact subset of $\mathbb{R}^d$. From \cite[Lemma 2.4]{fusco}  we have that $\mathcal{A}_\ell$ satisfies the first condition in \eqref{ellipticity_approx} with $s$ replaced by $s_\ell$, where $s_\ell\to 0$ as $\ell\to \infty$.
	In addition, since
	$$
	\dfrac{1}{\tilde{\alpha}}(|z|_g^2+s_\ell^2)^{\frac{p-2}{2}}|\xi|_g^2\leq g(\nabla  \mathcal{A}_\ell(z)\xi,\xi)\,	
	$$
for any $z,\xi$ vector fields in the tangent space at $x$ and for some $\tilde{\alpha}>0$, we obtain that $\mathcal{A}_\ell$ satisfies also the second condition in \eqref{ellipticity_approx}.

Let $\Omega \subset \R^d$ be an open set such that $\overline \Omega \subset \R^d \setminus \{O\}$. We approximate $w$ in $\Omega$ by $w_\ell>0$ which are solutions to 
 	\begin{equation}\label{approx_1}
	\begin{cases}
	\frac{1}{|x|^{(n-d)}}\diver(|x|^{(n-d)} \mathcal{A}_\ell(\nabla w_\ell))=  - w^{q-1}  & \text{in } \Omega
	\\
	w_\ell=w & \text{on } \partial \Omega\,,
	\end{cases}
	\end{equation}
where $\mathcal A_\ell$ is as in \eqref{ajz}.

\begin{lemma} \label{lemma_regularityA}
	Let $0<r<R$, $d\geq 3$ and let $w_\ell$ satisfy \eqref{approx_1}. Then $\mathcal{A_\ell}(\nabla w_\ell)\in W^{1,2}_{\rm loc}(B_r,|x|^{n-d})$ and, by setting $W_\ell=\nabla \mathcal A_\ell (\nabla w_\ell)$, we have
\begin{multline} \label{caccioAreg}
c \int_{B_r}   |W_\ell|_g^2\, |x|^{n-d} dx \\ \leq  \int_{B_R}  \left(\frac{(n-d)^2}{|x|^2}    +   |\nabla \eta|_g^2  \right) |\mathcal{A_\ell}(\nabla  w_\ell)|_g^2 \, |x|^{n-d} dx  +   \int_{B_R} |\nabla (|x|^{n-d} w^{q-1} )|_g | \mathcal A_\ell(\nabla  w_\ell)|_g dx\,,
\end{multline}
for any $\eta \in C_0^\infty (B_R)$ such that $\eta=1$ in $B_r$, where $c$ does not depend on $\ell$.
\end{lemma}

\begin{proof}

	The  idea is to prove a weighted Caccioppoli-type inequality for $\mathcal A_\ell (\nabla w_\ell)$, with weight $|x|^{n-d}$. Since $w_\ell$ solves a non-degenerate equation with smooth coefficients, we have that $w_\ell$ is smooth and $w_\ell \in W^{2,2}_{loc}(\R^d)$. 

In order to simply the notation, in \eqref{approx_1} we set $\Phi(x)=w^{q-1}(x)$, $e^{-f}=|x|^{n-d}$ and we omit the dependency on $\ell$. With this notation, \eqref{approx_1} takes the form 
\begin{equation} \label{approx_1_noind}
e^{f}\diver(e^{-f} \mathcal{A}(\nabla w))=  - \Phi \,.
\end{equation}
Let  $\varphi \in C^\infty_0(B_{R})$ and $m\in \{1,\ldots,d\}$.  We multiply \eqref{approx_1_noind} by $e^{-f} \nabla_m \varphi$ and integrate over a ball $B_R \subset E$ to get
	\begin{equation*}
	\int_{B_R}\diver(e^{-f}\mathcal{A}(\nabla w)) \nabla_m \varphi\, dV_g =-\int_{B_R}e^{-f} \Phi(w)\nabla_m \varphi\, dV_g \,. 
	\end{equation*}
We also recall that $dV_g = \alpha^{-1} dx$. From the divergence theorem we have 
	\begin{equation}\label{previous}
	\begin{aligned}
	\int_{B_R}  g(e^{-f} \mathcal{A}(\nabla w), \nabla(\nabla_m \varphi))\, dx=\int_{B_R} e^{-f} \Phi(w)\nabla_m \varphi\, dx,
	\end{aligned}
	\end{equation}
	and, by integrating by parts, we get 
	\begin{equation}\label{11.16}
	\begin{aligned}
	\int_{B_R} g(\nabla_m(e^{-f} \mathcal{A}(\nabla w)), \nabla\varphi )\, dx=\int_{B_R} \nabla_m(e^{-f} \Phi(w)) \varphi\, dx\, .
	\end{aligned}
	\end{equation}

Now, we take a cut-off function $\eta\in C^{\infty}_0(B_{R})$ with $\eta=1$ in $B_r$, and for $m\in\{1,\dots,n\}$ we set $\varphi=\mathcal{A}^m(\nabla w)\eta^2$. Hence, by summing over $m$ and recalling \eqref{11.16}, we proved that 	
\begin{equation*}
\int_{B_R} g_{ij} \nabla_m (e^{-f} \mathcal{A}^i(\nabla  w))\nabla^j(\mathcal{A}^m(\nabla  w)\eta^2)\, dx =\int_{B_R}  \nabla_m(e^{-f} \Phi(w)) \mathcal{A}^m(\nabla w)\eta^2 dx \,,
	\end{equation*}
where we recall that all the repeated indexes are summed. Since
$
\nabla^j \mathcal{A}^m = g^{jk} \nabla_k \mathcal{A}^m$,  and
$$
g_{ij} \nabla_m \mathcal{A}^i \, \nabla^j \mathcal{A}^m = g_{ij}  g^{jk} \, \nabla_m  \mathcal A^i \, \nabla_k \mathcal{A}^m  = \delta_i^k  \, \nabla_m  \mathcal A^i \, \nabla_k \mathcal{A}^m = \, \nabla_m  \mathcal A^k \, \nabla_k \mathcal{A}^m  = \tr (W^2 ),
$$	
then we have
\begin{multline} \label{tazzina}
\int_{B_R} \eta^2 e^{-f} \tr W^2 dx + \int_{B_R}  \nabla_m e^{-f}  g_{ij} \mathcal{A}^i(\nabla  w) \nabla^j(\mathcal{A}^m(\nabla  w)\eta^2) dx \\ + \int_{B_R} e^{-f}  g_{ij} \nabla_m(\mathcal{A}^i(\nabla  w) )\mathcal{A}^m(\nabla  w) 2 \eta \nabla^j \eta  \, dx  =\int_{B_R}  \nabla_m(e^{-f} \Phi(w)) \mathcal{A}^m(\nabla w)\eta^2 dx \,,
\end{multline} 
where we recall that $W_i^j = \nabla_j \mathcal{A}^i(\nabla w)$.

The first term on the LHS in \eqref{tazzina} can be handled as in the proof of Theorem 4.1 in \cite[p. 681]{AKM}. More precisely, by using the properties of $\mathcal{A}$ and some tools from linear algebra, one can prove that there exists a constant $c$ (which is uniform with respect to the paramenter $\ell$ in the approximation $\mathcal A_\ell$ of $\mathcal A$) such that 
\begin{equation} \label{traceW2}
\tr W^2 \geq c\,  |W|_g^2 \,,
\end{equation}
where $|W|_g^2 =g_{ij}g^{ab}W_{a}^iW_{b}^j$ is the square of the Hilbert-Schmidt norm of $W$. We estimate the second and third integrals in \eqref{tazzina} as 
\begin{multline*}
\Big{|} \int_{B_R}  \nabla_m e^{-f}  g_{ij} \mathcal{A}^i(\nabla  w) \nabla^j(\mathcal{A}^m(\nabla  w)\eta^2) dx \Big{|} + \Big{|} \int_{B_R} e^{-f}  g_{ij} \nabla_m(\mathcal{A}^i(\nabla  w) )\mathcal{A}^m(\nabla  w) 2 \eta \nabla^j \eta  \, dx \Big{|}  \\ 
\leq  C \int_{B_R}  (|\nabla e^{-f}|_g  \eta^2 +   e^{-f} \eta |\nabla \eta|_g)  |\mathcal{A}(\nabla  w)|_g |W|_g dx + 2 \int_{B_R}  |\nabla e^{-f}|_g   |\mathcal{A}(\nabla  w)|_g^2  \eta |\nabla \eta|_g dx 
\end{multline*}
and the RHS in \eqref{tazzina} as
\begin{equation*}
\int_{B_R}  \nabla_m(e^{-f} \Phi(w)) \mathcal{A}^m(\nabla w)\eta^2 dx \leq \int_{B_R} |\nabla (e^{-f} \Phi (w))|_g | \mathcal A(\nabla  w)|_g dx \,, 
\end{equation*}
and, by combining these last two estimates and \eqref{traceW2}, from \eqref{tazzina} we obtain that
\begin{multline*}
\int_{B_R} \eta^2 e^{-f} |W|_g^2 dx \leq C \int_{B_R}  (|\nabla e^{-f}|_g  \eta^2 +   e^{-f} \eta |\nabla \eta|_g)  |\mathcal{A}(\nabla  w)|_g |W|_g dx \\ + 2 \int_{B_R}  |\nabla e^{-f}|_g   |\mathcal{A}(\nabla  w)|_g^2  \eta |\nabla \eta|_g dx 
 +   \int_{B_R} |\nabla (e^{-f} \Phi (w))|_g | \mathcal A(\nabla  w)|_g dx \,.
\end{multline*}
We use Young's inequality twice and we have
\begin{multline*}
c \int_{B_R} \eta^2  |W|_g^2\, e^{-f} dx \\ \leq  \int_{B_R}  \left(\frac{|\nabla e^{-f}|_g^2}{e^{-2f}}  \eta^2  +   |\nabla \eta|_g^2  \right) |\mathcal{A}(\nabla  w)|_g^2 \, e^{-f} dx  +   \int_{B_R} |\nabla (e^{-f} \Phi (w))|_g | \mathcal A(\nabla  w)|_g dx\,,
\end{multline*}
which completes the proof.
\end{proof}

We are ready to prove the main result of this section.

\begin{proposition} \label{prop_regularityA}
Let $d\geq 3$ and $p \leq n/2$. Let $w$ be a solution of \eqref{pb_w_Riem}. Then $\mathcal{A}(\nabla w)\in W^{1,2}_{\rm loc}(\R^d,|x|^{n-d})$.  
\end{proposition}

\begin{proof}
From Lemma \ref{lemma_regularityA} it is clear that it is enough to prove the result in $B_1$. Since the approximation argument used in the proof of Lemma \ref{lemma_regularityA} may present difficulties due to the degeneracy of the weight at the origin, we argue by scaling as already done in the proof of Lemma \ref{lemma_gradient_x}. 

Let $\Omega= B_{5/2} \setminus \overline B_{1/2}$. For any $\mu \in (0,1)$ we define 
$$
\zeta_\mu (x) := w(\mu x) \,,
$$ 
for any $x \in \Omega$, and hence $\zeta_\mu$ satisfies \eqref{eq_L_zeta} in $\Omega$ and $ \zeta_\mu$ and $|\nabla \zeta_\mu|_g$ are uniformly bounded independently of $\mu$. Now we apply Lemma \ref{lemma_regularityA} with $\Omega=  B_{5/2} \setminus \overline B_{1/2}$ and $\Omega'=B_2 \setminus \overline B_1$. Since $\zeta_\mu$ satisfies \eqref{eq_L_zeta}, we consider $w_\ell$ to be the solution to
	\begin{equation}\label{approx_100}
	\begin{cases}
	\frac{1}{|x|^{(n-d)}}\diver(|x|^{(n-d)} \mathcal{A}_\ell(\nabla w_\ell))=  -\mu^p \zeta_\mu^{q-1}  & \text{in } \Omega
	\\
	w_\ell=\zeta_\mu & \text{on } \partial \Omega\,.
	\end{cases}
	\end{equation}
From Lemma \ref{lemma_regularityA} we have that
\begin{multline} \label{dis_step1}
c \int_{\Omega'}  |W_{\ell} |_g^2\, |x|^{(n-d)} dx  \leq  \int_{\Omega}  \left(\frac{(n-d)^2}{|x|^2}  +   |\nabla \eta|_g^2  \right) |\mathcal{A_\ell}(\nabla  w_\ell)|_g^2 \,|x|^{(n-d)} dx  \\ +  \mu^p  \int_{\Omega} |\nabla (|x|^{(n-d)} \zeta_\mu^{q-1})|_g | \mathcal A_\ell(\nabla  w_\ell)|_g dx \,,
\end{multline}
where $c$ does not depends on $\ell$ and $\mu$. Since $d \geq 3$ and $w_\ell$ and $\zeta_\mu$ are bounded in $C^{1,\alpha}$ uniformly with respect to $\ell$, we obtain that
\begin{equation*}
 \int_{\Omega'}   |W_{\ell} |_g^2\, |x|^{(n-d)} dx  \leq C \,,
\end{equation*}
where $C$ does not depend on $\ell$ and $\mu$ (here we are assuming that $0 < \mu \leq 1$). 
This implies that $\nabla \mathcal A_\ell (\nabla w_\ell)$ is uniformly bounded in $L^2(\Omega', |x|^{n-d})$. Moreover we recall that from \cite{DiBenedetto,Tolksdorf} we have that $w_\ell$ are uniformly bounded and converge to $w$ in $C^{1,\alpha}_{loc} (\Omega')$. Since $\mathcal A_\ell(\nabla w_\ell)$ converges to some function $\hat{\mathcal A}$ weakly in $W^{1,2}_{loc}$, then $\hat{\mathcal A} = \mathcal A (\nabla w)$. In particular we have
\begin{equation} \label{Wlimited}
 \int_{B_2 \setminus B_1} |\nabla \mathcal A (\nabla \zeta_\mu(x))|^2 |x|^{n-d} dx \leq C \,,
\end{equation}
where $C$ does not depend on $\mu$.

Now, we consider a sequence of radii $r_j=2^{-j}$ and set
$$
I_j = \int_{B_{2r_j} \setminus B_{r_j}} |\nabla \mathcal A (\nabla w(x))|_g^2\, |x|^{(n-d)} dx \,.
$$
From the scaling properties of $\mathcal A (\nabla w)$, we have that
$$
\int_{B_{2\mu} \setminus B_\mu} |\nabla \mathcal A (\nabla w(x))|^2 |x|^{n-d} dx
= \mu^{n-2p} \int_{B_2 \setminus B_1} |\nabla \mathcal A (\nabla \zeta_\mu(y))|^2 |y|^{n-d} dy  \leq C \mu^{n-2p} \,,
$$
where the last inequality follows from \eqref{Wlimited}. By choosing $\mu=r_j$ we readily find that 
$$
\int_{B_1} |\nabla \mathcal A (\nabla w(x))|^2 |x|^{n-d} dx \leq \sum_{j=0}^{+\infty} I_j \leq C \sum_{j=0}^{+\infty} \frac{1}{2^{n-2p}} 
$$
which is bounded provided that $p < n/2$, which implies the assertion.
\end{proof}

\section{Proof of Theorems \ref{th_main_p} and \ref{main_thm_Riem}} \label{section_finalproof}
In this section we give the proof of  Theorems \ref{th_main_p} and \ref{main_thm_Riem}. More precisely, we prove Theorem \ref{main_thm_Riem} and then Theorem \ref{th_main_p} follows as a corollary. It will be convenient to work with $v$ which is given by
\begin{equation}\label{def_v_p_II}
v(x)=w(x)^{-\frac{q-p}{p}},
\end{equation}
where $w$ is the solution to \eqref{EL_w_p_II}. We mention that, in our approach, the function $v$ is the analogue of the so-called  pressure function in \cite{DEL}. Clearly, $v$ inherits some regularity properties from $w$ (and hence from $u$). In particular, $v$ is of class $C^{1,\gamma}_{loc}(\R^d \setminus \{O\})$ and it satisfies 
\begin{equation} \label{eq_v_II}
\mathcal{L} v =\frac{n(p-1)}{p}\frac{|\nabla v|_g^p}{v}+\left(\frac{p}{n-p}\right )^{p-1}\frac{1}{v},
\end{equation}
where 
\begin{equation*}
\mathcal{L} v :  = |x|^{-(n-d)}\diver (|x|^{n-d}|\nabla v|_g^{p-2}\nabla v) \,.
\end{equation*}
Moreover, from Proposition \ref{stime_Vetois} we have the following asymptotic estimates 
\begin{equation} \label{asympt_v}
C_1 |x|^{\frac{p}{p-1}} \leq v(x) \leq C_2 |x|^{\frac{p}{p-1}}  \,, \quad C_1 |x|^{\frac{1}{p-1}}  \leq |\nabla v(x)|_g \leq C_2 |x|^{\frac{1}{p-1}}  \,, \quad |\Hess\, v(x)|_g \leq C_2 |x|^{\frac{2-p}{p-1}},
\end{equation}
for some $0<C_1\leq C_2$ and every $x\in \R^d \backslash B_R$, with $R>0$ large enough.  
By setting 
$$
\mathcal A (\nabla v) = |\nabla v|_g^{p-2}\nabla v \,,
$$
from Proposition \ref{prop_regularityA} we also have that $ \mathcal A (\nabla v) \in W^{1,2}_{loc}(\R^d, |x|^{n-d}) $, and we set
\begin{equation} \label{W_def}
W_i^j:= \nabla_i \mathcal A^j(\nabla v)\,,
\end{equation}
for $i,j=1,\ldots,d$.

The proof of Theorem \ref{th_main_p} is based on the following lemma, which gives an integral version of the differential identity proved in Proposition \ref{prop_diff_id}.

\begin{lemma} \label{lemma_integral_ident}
Let $v$ and $W$ be given by \eqref{def_v_p_II} and \eqref{W_def}, respectively. The following integral identity
\begin{multline}	\label{int_pos}
\int_{\R^d} v^{1-n} \Big[ (\mathcal{L}  v)^2-\trace (W^2)- (n-1) \frac{2p-1}{p} \frac{|\nabla v|_g^p}{v} \mathcal{L}v + n(n-1) \frac{p-1}{p} \frac{|\nabla v|_g^{2p}}{v^2}  
\\ - \ric_g\big(\mathcal A(\nabla v),\mathcal A(\nabla v)\big) -  \Hess f \big(\mathcal A(\nabla v),\mathcal A(\nabla v)\big) \Big  ] |x|^{n-d}dx = 0 
	\end{multline}
holds.
\end{lemma}

\begin{proof}

This identity can be obtained by approximation starting from \eqref{ident_final_generale} in Proposition \ref{prop_diff_id} by choosing $\gamma =1 -n$. The approximation argument is analogous to the one given in the proof of Proposition \ref{prop_regularityA}. For this reason we omit the details of the approximation argument and we consider directly \eqref{ident_final} instead of \eqref{ident_final_generale}.

We multiply \eqref{ident_final} by $|x|^{n-d}$ and integrate in $dx$ in a ball of sufficiently large radius $R$. From the regularity properties of $v$ and $\mathcal{A}(\nabla v)$ and by using the divergence theorem, we have 
\begin{multline*}
\int_{\partial B_R} g \left( -e^{-f} v^\gamma \mathcal{A}^i(\nabla v) \nabla_i \mathcal{A}^j(\nabla v) + v^\gamma  e^{-f} \mathcal{A}^j(\nabla v) \mathcal{L}v  + \gamma\frac{p-1}{p} v^{\gamma-1} |\nabla v|_g^p\, e^{-f}\mathcal{A}^j(\nabla v) , \nu \right)  \\
	= \int_{B_R} \Big\{ v^{\gamma} \left[ (\mathcal{L} v)^2 - \trace (W^2) \right] + \gamma \frac{2p-1}{p} v^{\gamma-1} |\nabla v|_g^p \mathcal{L}v + \gamma(\gamma-1)\frac{p-1}{p} v^{\gamma-2} |\nabla v|_g^{2p} \\
	- v^{\gamma}  {\rm Ric}\big(\mathcal A(\nabla v),\mathcal A(\nabla v)\big) - v^{\gamma}  \Hess f \big(\mathcal A(\nabla v),\mathcal A(\nabla v)\big)\Big\} e^{-f} dx \,,
\end{multline*}
where $ e^{-f} = |x|^{n-d}$ and $\gamma = 1-n$.
By letting $R \to \infty$ and using the decay estimates \eqref{asympt_v}, one has that the boundary integral on the LHS vanishes at infinity and then we get \eqref{int_pos}.    
\end{proof}

We are now ready to prove our main result. 

\vspace{1em}

\begin{proof}[Proof of Theorems \ref{th_main_p} and \ref{main_thm_Riem}]
We prove Theorem \ref{main_thm_Riem} and then Theorem \ref{th_main_p} immediately follows by recalling that $u(x) = c w( |x|^{\alpha-1} x ) $ for some constant $c$. We multiply \eqref{eq_v_II} by $v^{-n}$ and integrate in $dV_g$ over $\R^d$. By using the divergence
	theorem and the decay estimates in Proposition \ref{stime_Vetois} we obtain that 
	\begin{equation} \label{eq_integr}
	-\frac{n}{2}  \int_{\R^d} |x|^{n-d}|\nabla v|_g^2 v^{-n-1}dx+\frac{2}{n-2}\int_{\R^d} |x|^{n-d}v^{-n-1}dx =0.
	\end{equation}	
	Now we use \eqref{eq_v_II} in \eqref{int_pos} and, by taking into account \eqref{eq_integr}, we find 
	\begin{multline}
	\label{int_pos2}
	\int_{\R^d}  v^{1-n} \Big ( \frac1n (\mathcal{L} v)^2-\trace (W^2)- |\nabla v|_g^{2(p-2)}\ric_g(\nabla v,\nabla v) - |\nabla v|_g^{2(p-2)}\Hess_g f  \Big  ) |x|^{n-d}dx= 0, 
	\end{multline}
where we recall that $W$ is given by \eqref{W_def}.

Now we show that the quantity inside the integral in \eqref{int_pos2} is non-negative almost everywhere. We mention that, since $v \in C^{1,\gamma}_{loc} $ outside the origin and $\mathcal A (\nabla v) \in W^{1,2}_{loc}$, the following calculations make sense almost everywhere. 

We notice that $\mathcal L v$ can be written as
\begin{equation}  \label{jon}
\mathcal L v = L_p v +  (n-d)   |\nabla v|_g^{p-2} s(v) \,,
\end{equation}
where
$$
L_p v = \diver( |\nabla v|_g^{p-2}\nabla v) = \diver (\mathcal{A}(\nabla v)) 
$$
is the usual $p-$Laplace operator and 
$$
s(v)=g(\nabla \log |x|, \nabla v)\,.
$$
We recall that by, using Newton's inequality (see \cite[Lemma 3.2]{CianchiSalani}), one has
\begin{equation} \label{Newtonineq}
\tr (W^2)\geq \frac 1 d (L_p v)^2\,
\end{equation}
and from \eqref{jon} we have  
	\begin{equation}\label{Newton_c}
	\tr (W^2)\geq \frac 1 d \left ((\mathcal{L} v)^2-2 (n-d)    |\nabla v|_g^{p-2}s(v) \mathcal{L} v + (n-d)^2  |\nabla v|_g^{2(p-2)}s(v)^2 \right ). 
	\end{equation}
Since from \eqref{Newton_c}
\begin{equation} \label{cisiamoquasi}
\frac 1 n(\mathcal{L} v)^2- \tr(W^2)\leq - \frac {n-d}{nd}(\mathcal{L} v)^2+\frac {2 (n-d)}{d} |\nabla v|_g^{p-2} s(v) \mathcal{L} v -\frac {(n-d)^2}{d} |\nabla v|_g^{2(p-2)}  s(v)^2 \,. \\
\end{equation}
Now we distinguish two cases: $n=d$ and $n > d$ and prove that 
\begin{equation} \label{conclusion}
W_i^j=\lambda(x) \delta_i^j \quad \text{for a.e. }x\in \R^d,
	\end{equation}
$i,j \in \{1,\ldots,d\}$, for some function $\lambda: \R^d\to \R$. 

\noindent \emph{Case 1: $n=d$}. In this case \eqref{cisiamoquasi} implies that 
$$
\dfrac 1 d(\mathcal{L} v)^2- \tr(W^2)    \leq 0
$$ 
and from \eqref{int_pos2} we obtain that the equality sign holds in Newton's inequality and then \eqref{conclusion} follows.

\noindent \emph{Case 2: $n>d$}. From \eqref{cisiamoquasi} and by using 
\begin{equation} \label{disugnd}
2 |\nabla v|_g^{p-2}s(v)\mathcal{L} v \leq n|\nabla v|_g^{2(p-2)}s(v)^2+\frac 1 n (\mathcal{L}v)^2 \,,
\end{equation}
we get 
\begin{equation} \label{2}
\frac 1 n(\mathcal{L} v)^2- \tr(W^2)\leq (n-d) |\nabla v|_g^{2(p-2)}s(v)^2 \,.
\end{equation}
From \eqref{2}, \eqref{ricci} and \eqref{hessiano}, we find that  
\begin{multline*}
	\frac 1 n(\mathcal{L} v)^2- \tr(W^2)- \ric_g(\mathcal A(\nabla v),\mathcal A(\nabla v))- \Hess f  (\mathcal A(\nabla v),\mathcal A(\nabla v)) \leq\\ \frac{2-d+\alpha^2(n-2)}{|x|^2}|\nabla v|_g^{2(p-2)}\left (|\nabla v|^2-\frac{(\nabla v\cdot x)^2}{|x|^2}\right )
\end{multline*}
and then, by using condition \eqref{alpha_cond}
\begin{equation*} 
	\alpha^2\leq \frac{d-2}{n-2},
	\end{equation*}
we find
\begin{equation*}
\frac 1 n(\mathcal{L} v)^2- \tr(W^2)- \ric_g(\mathcal A(\nabla v),\mathcal A(\nabla v))- \Hess_g f  (\mathcal A(\nabla v),\mathcal A(\nabla v))  \leq 0 \,.
\end{equation*}	
This last inequality, which holds a.e. in $\mathbb{\R}^d$, and \eqref{int_pos2} imply that the equality sign must hold, and hence all the inequalities in this proof are actually equalities. In particular, the equality sign must hold in the following two inequalities: 
$$
\tr (W^2) \geq \frac 1 d (L_p v)^2\,,
$$
and
$$
2 |\nabla v|_g^{p-2}s(v)\mathcal{L} v \leq n|\nabla v|_g^{2(p-2)}s(v)^2+\frac 1 n (\mathcal{L}v)^2 \,,
$$
and then we must have that \eqref{conclusion} holds. We notice that in this case, the last inequality also implies that
\begin{equation} \label{ndivd}
\mathcal L v = n |\nabla v|_g^{p-2} s(v) \,.
\end{equation}
%
%

Hence, in both Cases 1 and 2 we have that \eqref{conclusion} holds and then
$$
L_p v = \tr W =   \lambda (x) d \,, 
$$ 

\noindent We notice that from \eqref{eq_v_II} and \eqref{jon} we have that  
\begin{equation*}   
\lambda (x) d  = \frac{n(p-1)}{p}\frac{|\nabla v|_g^p}{v}+\left(\frac{p}{n-p}\right )^{p-1}\frac{1}{v} -   (n-d)   |\nabla v|_g^{p-2} s(v)\,.
\end{equation*}

According to the regularity of $v$, we know that $\lambda$ is locally of class $C^\gamma$ in $\R^d \setminus\{O\}$ and, from elliptic regularity theory, $\lambda$ is $C^{1,\gamma}$ at points where $\nabla v \neq 0$ in $\R^d \setminus\{O\}$. Thanks to the regularity of $\lambda$ at points where $\nabla v \neq 0$, for any fixed $i,j\in \{1,\ldots,d\}$ we have that (in the following formula repeated indices are not summed)
$$
\nabla_j \lambda (x) = \nabla_j \nabla_i \mathcal A^i (\nabla v(x)) = \nabla_i \nabla_j \mathcal A^i (\nabla v(x)) - R_{jik}^i \mathcal A^k(\nabla v(x)) = \nabla_i ( \lambda(x) \delta_j^i)   - R_{jik}^i \mathcal A^k(\nabla v(x))   \,,
$$
where we used \eqref{conclusion}.
By summing over $i=1,\ldots, d$ , we obtain that 
$$
(d-1) \nabla_j \lambda (x) = - R_{jk} \mathcal A^k(\nabla v(x))  \,,
$$
where $R_{jk}$ are the components of the Ricci tensor given by \eqref{Ricci_jk}.
Hence we have that
$$
 R_{jk} \mathcal A^k(\nabla v(x)) = (1-\alpha^2) \frac{d-2}{|x|^2} \left( \mathcal A_j - \left(\mathcal A \cdot \frac{x}{|x|} \right) \frac{x_j}{|x|}  \right)\,, 
$$
the last two equations imply that 
$$
 \nabla_j \lambda (x) = (\alpha^2-1) \frac{d-2}{d-1} \frac{1}{|x|^2} \left( \mathcal A_j - \left(\mathcal A \cdot \frac{x}{|x|} \right) \frac{x_j}{|x|}  \right)\,  \,,
$$
and then $\nabla \lambda$ does not have radial components. Hence we obtain that on each connected component of $\{\nabla v \neq 0\}$ either $\lambda$ is constant (for $\alpha=1$) or $\lambda$ is zero-homogeneous (for $\alpha <1$). This property actually holds in the whole space, since $\lambda$ is continuous and $\{\nabla v \neq 0\}$ has no interior points.

If $\alpha=1$ (i.e. $a=b$) then the fact that $\lambda$ is constant implies that $v(x) = c_1 + c_2 |x-x_0|^{\frac{p}{p-1}}$. If $\alpha < 1$ then we have to work a little more. In this case we know that $\lambda$ is zero-homogeneous and hence from \eqref{conclusion} we have that $\nabla_j \mathcal A^i (\nabla v)$ is zero homogeneous for any fixed $i,j=1,\ldots,d$. This implies that $\mathcal A (\nabla v)$ is one-homogeneous up to an additive constant. Since 
$$
\mathcal L v=  \tr(W) +  (n-d)   |\nabla v|_g^{p-2} s(v)
$$ 
and from \eqref{ndivd} we obtain that $\mathcal L v$ is zero-homogenous. By using \eqref{eq_v_II} we have that $v f_0 =|f_1 + x_0|^{\frac{p}{p-1}} + c$, where $c$ is a constant, $x_0 \in \R^d$, $f_0:= \mathcal L v$ is a zero-homogeneous function and $f_1$ is a vectorial one-homogeneous function. By letting $x \to \infty$ and using Proposition \ref{stime_Vetois}, we obtain that $f_1$ is constant in the angular direction. Since $v$ is continuous at the origin (see \cite{ColoradoPeral}), we obtain that also $f_0$ is constant. This implies that
$$
v(x) = A |x+x_0|^{\frac{p}{p-1}} + B 
$$
for some constant $A$ and $B$ and some $x_0 \in \R^d$. By using this expression in \eqref{eq_v_II} we conclude.
\end{proof}

\medskip

\section*{Acknowledgments} 
\noindent The authors thank Luigi Vezzoni for useful discussions and remarks.
The authors have been partially supported by the ``Gruppo Nazionale per l'Analisi Matematica, la Probabilit\`a e le loro Applicazioni'' (GNAMPA) of the ``Istituto Nazionale di Alta Matematica'' (INdAM, Italy). R.C. has been partially supported by the PRIN 2017 project ``Qualitative and quantitative aspects of nonlinear PDEs''.

\end{document}